\documentclass[a4paper,reqno]{amsart}

\usepackage{amsmath}
\usepackage{paralist}
\usepackage{graphics}
\usepackage{epsfig}
\usepackage{amsfonts}
\usepackage{amssymb}
\usepackage{hyperref}
\usepackage{epstopdf}
\usepackage{color}

\newtheorem{theorem}{Theorem}[section]
\newtheorem{lemma}[theorem]{Lemma}

\def\ifl{\iffalse }

\def\bc{\begin{center}}       \def\ec{\end{center}}
\def\ba{\begin{array}}        \def\ea{\end{array}}
\def\be{\begin{equation}}     \def\ee{\end{equation}}
\def\bea{\begin{eqnarray}}    \def\eea{\end{eqnarray}}
\def\beaa{\begin{eqnarray*}}  \def\eeaa{\end{eqnarray*}}

\numberwithin{equation}{section}

\newtheorem{definition}[theorem]{Definition}

\numberwithin{equation}{section}

\begin{document}

\title[Two-species chemotaxis model with signal absorption]
{Global solvability  and asymptotical behavior in a two-species chemotaxis model with signal absorption}

\author{Guoqiang Ren}
\address{School of Mathematics and Statistics, Huazhong University of Science and Technology, Wuhan, 430074, Hubei, P. R. China;  Hubei Key Laboratory of Engineering Modeling and Scientific Computing,
       Huazhong University of Science and Technology,
       Wuhan, 430074, Hubei, P. R. China}
\email{597746385@qq.com}

\author{Tian Xiang$^*$}
\address{Institute for Mathematical Sciences, Renmin University of China, Bejing, 100872, China}
\email{txiang@ruc.edu.cn}
\thanks{$^*$ Corresponding author.}

\subjclass[2000]{Primary:  35K59, 35B65, 35B40, 35A09, 35K51; Secondary: 35A01, 92D25.}


\keywords{Two-species chemotaxis model, signal absorption,  global existence, boundedness,  asymptotics.}

\begin{abstract}

 In this work, we study global existence, eventual smoothness and asymptotical behavior   of positive solutions for the following two-species chemotaxis consumption model:
$$\left\{ \begin{array}{lll}
&u_t = \Delta u-\chi_1\nabla \cdot ( u\nabla w),  &\quad x\in \Omega, t>0, \\[0.2cm]
& v_t =\Delta v-\chi_2\nabla \cdot (u\nabla w),  &\quad x\in \Omega, t>0,  \\[0.2cm]
& w_t =\Delta w -(\alpha u+\beta v)w, &\quad x\in \Omega, t>0,
  \end{array}\right.
$$
 in a  bounded smooth but not necessarily  \it{convex} domain  $\Omega\subset \mathbb{R}^n (n=2,3,4,5)$ with nonnegative initial data $u_0, v_0, w_0$ and homogeneous Neumann boundary data. Here, the parameters $\chi_1,\chi_2$ are positive  and $\alpha,\beta$ are nonnegative.

Under a smallness condition $\max\{\chi_1,\chi_2\}\|w_0\|_{L^\infty}<\pi\sqrt{2/n}$, boundedness of classical solutions and stabilization to constant equilibrium have been shown in \cite{ZT19}. Here, without any smallness condition, we show  global existence and uniform-in-time boundedness of classical solutions  in  2D  and   global existence, eventual smoothness  and asymptotical behavior (in convex domains)  of weak solutions in nD (n=3,4,5). Our findings also extend and improve the one-species chemotaxis-consumption model studied in  \cite{Ren-Liu-4, TW12-JDE}.
 \end{abstract}

\maketitle

\section{Introduction and sketch of the main results}

In this project, we investigate the following Neumann initial-boundary value problem for a
two-species chemotaxis system with consumption of chemoattractant:
\be \label{PPT}\begin{cases}
u_t = \Delta u-\chi_1 \nabla \cdot (u\nabla w),  &x\in \Omega, t>0, \\[0.2cm]
v_t = \Delta v-\chi_2 \nabla \cdot (v\nabla w), & x\in \Omega, t>0, \\[0.2cm]
 w_t = \Delta w -(\alpha u+\beta v)w,  & x\in \Omega, t>0, \\[0.2cm]
\frac{\partial u}{\partial \nu}=\frac{\partial v}{\partial \nu}=\frac{\partial w}{\partial \nu}=0, & x\in \partial \Omega, t>0,\\[0.2cm]
u(x,0)=u_0(x) \ , v(x,0)=v_0(x) \ ,  w(x,0)=w_0(x) \ , & x\in \Omega.  \end{cases}  \ee
Hereafter,   $\Omega \subset \mathbb{R}^n (n\geq 1)$ is a bounded domain with a smooth boundary $\partial \Omega$ and $\frac{\partial}{\partial\nu}$ denotes the outer normal derivative; the unknown variables $u=u(x,t)$ and $v=v(x,t)$  denote  the population densities of two species and $w$ represents the concentration of the chemoattractant,  $\chi_1, \chi_2, \alpha$ and $\beta$ are positive constants and the given initial data  are conveniently assumed throughout this paper to satisfy, for some $r>\max\{2,n\}$, that
\be \label{initial-data-reg}
(u_0,v_0,w_0)\in C^0(\overline{\Omega})\times C^0(\overline{\Omega})\times W^{1,r}(\Omega), \ \  u_0, v_0, w_0\geq, \not\equiv0.  \ee
The model \eqref{PPT} is used  in mathematical biology to account  the biased movement of two populations in respond to the concentration gradient of one common chemical signal. It is an obvious extension of the well-known  Keller-Segel second model:
 \be \label{KS}\begin{cases}
u_t = \Delta u-\chi \nabla \cdot (u\nabla w),  &x\in \Omega, t>0, \\[0.2cm]
 w_t =  \Delta w-uw,  & x\in \Omega, t>0,   \end{cases}  \ee
which and its variants have been studied mathematically in various contexts. Speaking of classical solutions, if
 \be\label{small-w0}
 \text{either} \  \ \   n\leq 2 \ \ \ \text{or} \ \ \   \chi\|w_0\|_{L^\infty(\Omega)}\leq \frac{1}{6(n+1)},
 \ee
 then global  boundedness of  the solution $(u,w)$  to \eqref{KS} is ensured   and further  any global bounded such solution  converges uniformly according to
\be\label{uw-cov}
 \lim_{t\rightarrow \infty} \left(\left\|u(\cdot, t)-\bar{u}_0\right\|_{L^\infty(\Omega)}+\left\|w(\cdot, t)\right\|_{L^\infty(\Omega)}\right)=0, \ \ \ \bar{u}_0=\frac{1}{|\Omega|}\int_\Omega u_0,
 \ee
cf.  \cite{Tao11, Xiang18-NA, ZL17-JMP}. Without smallness condition like \eqref{small-w0}, the problem possesses at least one global (certain) weak solution which is eventually smooth and enjoys the convergence property \eqref{uw-cov} in 3D bounded convex domains \cite{TW12-JDE}. As a simple starting motivation, we are wondering  whether such type weak solution continues to exist in $4$D  or higher and, if so, whether it also enjoys \eqref{uw-cov}. So far,  it is still widely open whether \eqref{KS} possesses blow-ups in higher dimensions,  only certain blow-up properties of  the local classical solutions to \eqref{KS} are recently known \cite{JWZ18-JDE}. For studies on chemotaxis-consumption systems with different boundary conditions, we refer the interested reader to the very recent works \cite{FLM20,LWY20}. For properties of solutions in chemotaxis-consumption type models in more complex framework, for instance, with tensor-valued sensitivity,  singular sensitivity,  logistic source, predator-prey interaction or  fluid interaction etc, we refer the interested reader to
\cite{BK17-M3AS, BBTW15, JW16, LW17-DCDS, Win15-SIAM, Win16-M3AS, Win17-JDE} and the references therein.

It is well-known that logistic type source has an effective role in enhancing global existence, and boundedness in chemotaxis-involving systems. Indeed, a lot of studies have been done to the IBVP \eqref{PPT} with Lotka-Volterra type competitive kinetics (the same boundary and initial conditions are suspended):
\be \label{Log-PPT}\begin{cases}
u_t = \Delta u-\chi_1 \nabla \cdot (u\nabla w)+\mu_1u\left(1-u-a_1v\right),  &x\in \Omega, t>0, \\[0.2cm]
v_t = \Delta v-\chi_2 \nabla \cdot (v\nabla w)+\mu_2v\left(1-v-a_2u\right), & x\in \Omega, t>0, \\[0.2cm]
 w_t = \Delta w -(\alpha u+\beta v)w,  & x\in \Omega, t>0.  \end{cases}  \ee
Global  boundedness of classical solutions to \eqref{Log-PPT} are guaranteed under
 $$
 \text{either} \  \ \   n\leq 2 \ \ \ \text{or} \  \  \   \max\left\{\chi_1, \  \chi_2\right\}\|w_0\|_{L^\infty(\Omega)}<\frac{\pi}{\sqrt{n+1}}.
$$
Such bounded solutions are known (\cite{HKMT-17, JX19,  Ren-Liu-2, Ren-Liu-3, WMHZ-18}) to stabilize according to
\be\label{conv-log}
\left(u(\cdot, t), v(\cdot, t), w(\cdot, t)\right)\overset{\text{ in } L^\infty(\Omega)}\longrightarrow \begin{cases}  \left(\frac{1-a_1}{1-a_1a_2},  \frac{1-a_2}{1-a_1a_2}, 0\right),  &\text{if } a_1, a_2\in(0,1), \\[0.2cm]
\left(0,1, 0\right), &\text{if } a_1\geq 1> a_2>0, \\[0.2cm]
\left(1, 0,  0\right), &\text{if } 0<a_1< 1\leq a_2.
\end{cases}
\ee
Recently, global existence of generalized weak solutions and their long time  behaviors (similar to \eqref{conv-log}) to \eqref{Log-PPT} in nD are shown in \cite{Ren-Liu-4} under
$$
   \max\left\{\chi_1, \  \chi_2\right\}\|w_0\|_{L^\infty(\Omega)}<\frac{1}{2}.
$$
Indeed,  fluid  interaction has been incorporated in \eqref{Log-PPT}, cf.  \cite{HKMT-17, JX19}. We  also mention that single or multiple-species signal-production type chemotaxis systems with/without fluid interaction have been widely investigated  also e.g. in
\cite{BW16,CKM, LW19-JMAA, LX20, LR20-JKMS, TW12-non,  TMQ20, Win13, TX18-JMAA} and the references therein.

Now,  to formulate our main motivation of this project, we observe, without any damping source, global existence and boundedness of classical solutions to the IBVP \eqref{PPT} in $n$D  are obtained in \cite{ZT19}  under
\be\label{ZTcond}
\max\left\{\chi_1, \  \chi_2\right\}\|w_0\|_{L^\infty(\Omega)}<\pi\sqrt{\frac{2}{n}}.
\ee
Under this smallness condition,  stabilization of  solutions is also naturally derived:
\be\label{uvw-cov-zt}
 \lim_{t\rightarrow \infty} \left(\left\|u(\cdot, t)-\bar{u}_0\right\|_{L^\infty(\Omega)}+\left\|v(\cdot, t)-\bar{v}_0\right\|_{L^\infty(\Omega)}+\left\|w(\cdot, t)\right\|_{L^\infty(\Omega)}\right)=0.
 \ee
Comparing  these existing results, we find, even in the presence of Lotka-Volterra type competitive kinetics, certain smallness condition on initial data still needs to be  imposed to have global existence, boundedness and convergence. A natural question arises: whether and how far can we solve \eqref{PPT} globally without any damping source? More specifically, inspired from \eqref{small-w0} and \eqref{ZTcond}, we are wondering, first,
\begin{itemize}
\item[(Q1)]
without  the smallness condition in  \eqref{ZTcond} with $n=2$, can we  still have $2$D  global existence and boundedness of classical solutions to the IBVP \eqref{PPT}?
\end{itemize}
Second, based on the existing knowledge about the one-species chemotaxis consumption model \eqref{KS}, cf. \cite{Ren-Liu-4, TW12-JDE, Win17-JDE} and the references therein, we are wondering
\begin{itemize}
\item[(Q2)]without the smallness condition in  \eqref{ZTcond}, how far  can we solve the two-species chemotaxis-consumption model \eqref{PPT}  globally in a weak solution sense in $\geq3$D and, if so, how do such weak solutions behave after certain perhaps long waiting  time?
\end{itemize}
In this work, we shall answer (Q1) and (Q2) in a positive way for the IBVP  \eqref{PPT}: we show, without any smallness condition, global existence and uniform-in-time boundedness of classical solutions  in  2D  and   global existence, eventual smoothness and asymptotical behavior (in convex domains)  of weak solutions   in $n$D ($n=3,4,5$).
\begin{theorem}[\textbf{Global dynamics  for \eqref{PPT}}]\label{global-ext-3d-4d-5d}  Let $\chi_1,\chi_2, \alpha,\beta>0$ and $\Omega\subset\mathbb{R}^n$ $(n\leq 5)$ be a bounded and smooth domain, and let  $u_0,v_0$ and $w_0$ fulfill  \eqref{initial-data-reg}.
\begin{itemize}
\item[(B1)] [\textbf{Global boundedness and convergence in $2$D}] When $n=2$, the IBVP \eqref{PPT} has a unique   global classical solution which is bounded on $\Omega\times(0,\infty)$ in the sense there exists $C>0$ such that
    \be\label{linfty-bdd}
    \|u(\cdot, t)\|_{L^\infty(\Omega)}+\|v(\cdot, t)\|_{L^\infty(\Omega)}+\|w(\cdot, t)\|_{W^{1,\infty}(\Omega)}\leq C, \ \  \ \forall t>0.
    \ee
   Naturally, such bounded solution converges according to \eqref{uvw-cov-zt}.
\item[(B2)] [\textbf{Global existence of weak solutions in  $3,4,5$D}] When $n=3,4,5$,  there exists at least one triple $(u,v,w)$ of nonnegative functions satisfying
\be \label{data-reg}\begin{cases}
u\in L^{\frac{n+2}{n}}_{loc}(\overline{\Omega}\times[0,\infty))\cap L^{\frac{n+2}{n+1}}_{loc}([0,\infty);W^{1,\frac{n+2}{n+1}}(\Omega)),\\[0.2cm]
v\in L^{\frac{n+2}{n}}_{loc}(\overline{\Omega}\times[0,\infty))\cap L^{\frac{n+2}{n+1}}_{loc}([0,\infty);W^{1,\frac{n+2}{n+1}}(\Omega))\mbox{ and }\\[0.2cm]
w\in L^4_{loc}([0,\infty);W^{1,4}(\Omega)),
\end{cases}  \ee
which are a global weak solution of \eqref{PPT} in the sense of Definition \ref{weak-uvw} below.
\item[(B3)][\textbf{Eventual smoothness and convergence in  convex domains}] When $\Omega\subset \mathbb{R}^n (n=3,4,5)$ is a smooth,  bounded and convex domain, there exists $T^{\ast}>0$ such that the global weak  solution obtained in (B2) is bounded,  belongs to $C^{2,1}(\overline{\Omega}\times[T^{\ast},\infty))$ and converges according to  \eqref{uvw-cov-zt}.
\end{itemize}
\end{theorem}

In 2D setting, our global boundedness of classical solutions  in  (B1) removes the smallness  condition  \eqref{ZTcond} with $n=2$ as required in \cite{ZT19}. In 3D setting, our   global existence weak solutions relaxes the commonly used convexity assumption on $\Omega$   in the literature, cf. \cite{Ren-Liu-4, TW12-JDE, Win17-JDE}. Moreover, the known eventual smoothness  and large time  behavior    of weak solutions  in 3D convex domains has been extended to 4 and 5D and our findings also extend and improve the one-species chemotaxis-consumption model studied in  \cite{Ren-Liu-4, TW12-JDE}.

The layout of this paper is structured as follows: In Sect. 2, we combine and extend existing technique to study local and global well-posedness  with focus on boundedness and convergence in 2D for \eqref{PPT}, which relies on a crucial evolution identity \eqref{ulnu+gradv/v-fin}. In Sect. 3, we formulate  the approximating  system of \eqref{PPT},  introduce the concept of weak solutions and derive basic properties of approximating solutions. In Sect. 4, we  motivate  and extend arguments mainly from  \cite{TW12-JDE, Win17-JDE} to derive global existence, eventual smoothness  and convergence (in convex domains)  of weak solutions in 3, 4 and 5D, as detailed in Subsect. 4.2 and Subsect. 4.3. The convexity of domain $\Omega$ could  be removed mainly because the boundary integral emerging  from the identity \eqref{ulnu+gradv/v-fin} can be properly controlled in a manner as in \eqref{boundary-esti-fin}.

\section{Boundedness and convergence in 2D}
\subsection{Basic facts and Local existence}
In the subsequent analysis, we shall need the well-known Gagliardo-Nirenberg interpolation inequality, we list it here for convenience of reference.
\begin{lemma}\label{GN-inter}(Gagliardo-Nirenberg interpolation inequality \cite{Fried,  Nirenberg66}) Let $p\geq 1$ and  $q\in (0,p)$. Then there exists a positive constant  $C_{GN}=C_{p,q}$ such that
 $$
 \|w\|_{L^p(\Omega)} \leq C_{GN}\Bigr(\|\nabla w\|_{L^2(\Omega)}^{\delta}\|w\|_{L^q(\Omega)}^{(1-\delta) }+\|w\|_{L^r(\Omega)}\Bigr), \quad \forall w\in H^1(\Omega)\cap L^q(\Omega),
 $$
where $r>0$ is arbitrary and  $\delta$ is given by
 $$
 \frac{1}{p}=\delta(\frac{1}{2}-\frac{1}{n})+\frac{1-\delta}{q}\Longleftrightarrow \delta=\frac{\frac{1}{q}-\frac{1}{p}}{\frac{1}{q}+\frac{1}{n}-\frac{1}{2}}\in(0,1).
 $$
\end{lemma}
Next, we state the following well-established local  solvability,    extendibility  and basic estimates  of solutions to  the  IBVP   \eqref{PPT}.
\begin{lemma}\label{local-in-time}Let $\chi_1,\chi_2, \alpha,\beta>0$ and $\Omega\subset\mathbb{R}^n$ $(n\geq 1)$ be a bounded and smooth domain, and let  $u_0,v_0$ and $w_0$ fulfill  \eqref{initial-data-reg} with $r>\max\{2, n\}$. Then there is a unique,  positive and  classical maximal solution $(u,v,w )$ of the IBVP \eqref{PPT} on some maximal interval $[0, T_m)$ with $0<T_m \leq \infty$ such that
\begin{align*}
&(u,v)\in \left(C\left(\overline{\Omega}\times [0, T_m)\right)\cap C^{2,1}(\overline{\Omega}\times (0, T_m))\right)^2, \\
 &w\in C\left(\overline{\Omega}\times [0, T_m)\right)\cap C^{2,1}(\overline{\Omega}\times (0, T_m))\cap W_{\text{loc}}^{1,\infty}((0,T_m), W^{1,r}(\Omega)).
\end{align*}
 If  $T_m<\infty$, then the following extensibility criterion holds:
$$
 \limsup_{t\nearrow T_m}\left(\|u(\cdot,t)\|_{L^{\infty}(\Omega)}+\|v(\cdot,t)\|_{L^{\infty}(\Omega)}
 +\|w(\cdot,t)\|_{W^{1,r}(\Omega)}\right)=\infty.
$$
Furthermore, $u$ and $v$ have conservation of mass within $(0, T_m)$:
\be\label{basic-est}
\|u(\cdot,t)\|_{L^1(\Omega)}=\|u_0\|_{L^1(\Omega)}, \ \  \  \  \|v(\cdot,t)\|_{L^1(\Omega)}=\|v_0\|_{L^1(\Omega)},
\ee
and, for any $p\in[1, \infty]$,  the $L^p$-norm of $w$ is non-increasing:
\be\label{reg-2.16-0}
t\mapsto\|w(\cdot,t)\|_{L^p(\Omega)} \text{ is  non-increasing in } [0,T_m);
\ee
in particular,
\be\label{reg-2.17-0}
\|w(\cdot,t)\|_{L^p(\Omega)}\leq\|w_0\|_{L^p(\Omega)}.
\ee
\end{lemma}
\begin{proof}The   local  existence, regularity and extendibility   of classical solutions to  the  IBVP   \eqref{PPT}  are based on contraction mapping argument and   parabolic regularity theory of parabolic systems,  which can be found in   \cite{BBTW15, Tao10, Tao11, Win10, Win12-CPDE,WSW16}.  The conservations of $u$ and $v$ in \eqref{basic-est} follows upon integration by parts due to the no-flux boundary conditions, and the positivity of solution,  \eqref{reg-2.16-0} and hence \eqref{reg-2.17-0} with $p=\infty$ follows from an application of the (strong) maximum principle. For the case $p\in[1,\infty)$, testing the $w$-equation by $w^{p-1}$ and integrating by parts, we get
$$
\frac{d}{dt}\int_\Omega w^p=-(p-1)\int_\Omega w^{p-2}|\nabla w|^2-\int_\Omega\left(\alpha u+\beta v\right)w\leq 0,
$$
which upon being integrated from $s$ to $t$ entails \eqref{reg-2.16-0}.
\end{proof}
Henceforth, we will denote   by $C_i$ various constants which may vary line by line.
\subsection{Uniform boundedness and global existence in 2D}
In this subsection, we shall extend the idea in \cite{Xiang18-NA} to show $2D$ boundedness and global existence of classical solutions to the IBVP \eqref{PPT} without any smallness condition, thus proving (B1). To gain our gaol, we first establish a series  of important a-priori estimates; we state them in $n$ dimensional setting since they are valid in arbitrary space dimensions and they are quite convenient in subsequent sections.

To move from $L^1$-boundness  obtained in \eqref{basic-est} and \eqref{reg-2.16-0} to  higher order $L^p$-regularity, we first compute from the $u$- and $v$-equations in   \eqref{PPT} that
\be\label{ulnu-id}
\begin{cases}
\frac{d}{dt}\int_\Omega u\ln u+\int_\Omega \frac{|\nabla u|^2}{u}=\chi_1\int_\Omega \nabla u\nabla w,\ \ \ \  \forall t\in (0, T_m), \\[0.2cm]
\frac{d}{dt}\int_\Omega v\ln v+\int_\Omega \frac{|\nabla v|^2}{v}=\chi_2\int_\Omega \nabla v\nabla w, \ \ \ \ \forall t\in (0, T_m).
\end{cases}
\ee
To cancel out exactly the chemotaxis involving terms  on the right-hand side of  \eqref{ulnu-id},  inspired from \cite{Win12-CPDE, Xiang18-NA}, we compute the following time evolution:
\begin{align*}
\frac{1}{2}\frac{d}{dt}\int_\Omega \frac{|\nabla w|^2}{w}=\int_\Omega \frac{\nabla w\nabla w_t}{w}-\frac{1}{2}\int_\Omega \frac{|\nabla w|^2}{w^2}w_t.
\end{align*}
For the first term, an integration by parts  along with the $w$-equation entails that
\begin{align*}
\int_\Omega \frac{\nabla w\nabla w_t}{w}&=\int_\Omega \frac{1}{w}\nabla w\cdot\nabla \Delta w- \int_\Omega \frac{\left(\alpha u+\beta v\right)}{w}|\nabla w|^2- \int_\Omega \left(\alpha\nabla u+\beta \nabla v\right)\nabla w\\
&=\frac{1}{2}\int_{\partial\Omega} \frac{1}{w} \frac{\partial |\nabla w|^2}{\partial\nu}+\frac{1}{2}\int_\Omega \frac{\nabla w\cdot\nabla|\nabla w|^2}{w^2}-\int_\Omega \frac{|D^2w|^2}{w}\\
&\ \ \ \ - \int_\Omega \frac{\left(\alpha u+\beta v\right)}{w}|\nabla w|^2- \alpha\int_\Omega  \nabla u \nabla w-\beta\int_\Omega  \nabla v \nabla w,
\end{align*}
where  we have applied the following point-wise identity
$$
2\nabla  w\cdot \nabla \Delta w=\Delta|\nabla w|^2-2|D^2w|^2, \quad |D^2w|^2=\sum_{i,j=1}^n|w_{x_ix_j}|^2.
$$
In the same sprit, for the second term, one has
\begin{align*}
\frac{1}{2}\int_\Omega \frac{|\nabla w|^2}{w^2}w_t=\int_\Omega \frac{|\nabla w|^4}{w^3}-\frac{1}{2}\int_\Omega \frac{\nabla w\cdot\nabla |\nabla w|^2}{w^2}-\frac{1}{2}\int_\Omega \frac{\left(\alpha u+\beta v\right)}{w}|\nabla w|^2.
\end{align*}
Collecting  these equalities together, we end up with
\be\label{gradv/v}
\begin{split}
&\frac{1}{2}\frac{d}{dt}\int_\Omega \frac{|\nabla w|^2}{w}+\left(\int_\Omega \frac{|D^2 w|^2}{w}- \int_\Omega \frac{\nabla w\cdot\nabla|\nabla w|^2}{w^2} +\int_\Omega \frac{|\nabla w|^4}{w^3}\right)\\
&=\frac{1}{2}\int_{\partial\Omega} \frac{1}{w} \frac{\partial |\nabla w|^2}{\partial\nu}-\frac{1}{2}\int_\Omega \frac{\left(\alpha u+\beta v\right)}{w}|\nabla w|^2- \alpha \int_\Omega \nabla u\nabla w-\beta \int_\Omega \nabla v\nabla w.
\end{split}
\ee
For later use, cf. Lemma \ref{ulnu+gradv2-bdd-lem}, it is a good place to  notice   (cf. \cite[(3.9)]{Xiang18-NA}) that
\be\label{vD^2}
 \int_\Omega w|D^2\ln w|^2=\int_\Omega \frac{|D^2 w|^2}{w}- 2\int_\Omega \frac{ \nabla w\cdot D^2 w\cdot \nabla w}{w^2} +\int_\Omega \frac{|\nabla w|^4}{w^3}.
\ee
Then the elementary inequality $-2ab\geq -\frac{1}{2}a^2-2b^2$  implies
\be\label{vd2-v3}
\int_\Omega w|D^2\ln w|^2\geq \frac{1}{2}\int_\Omega \frac{|D^2 w|^2}{w}-\int_\Omega \frac{|\nabla w|^4}{w^3}.
\ee
Plugging \eqref{vD^2}  into \eqref{gradv/v} and noticing $\nabla w\nabla|\nabla w|^2=2\nabla w\cdot D^2 w\cdot \nabla w$, we get  that
\be\label{gradv/v-fin}
\begin{split}
&\frac{1}{2}\frac{d}{dt}\int_\Omega \frac{|\nabla w|^2}{w}+\int_\Omega w|D^2\ln w|^2+\frac{1}{2}\int_\Omega \frac{\left(\alpha u+\beta v\right)}{w}|\nabla w|^2\\
&=\frac{1}{2}\int_{\partial\Omega} \frac{1}{w} \frac{\partial |\nabla w|^2}{\partial\nu} - \alpha \int_\Omega \nabla u\nabla w-\beta \int_\Omega \nabla v\nabla w.
\end{split}
\ee
Combining \eqref{ulnu-id} with  \eqref{gradv/v-fin},  we conclude an important identity for \eqref{PPT} as follows:
\be\label{ulnu+gradv/v-fin}
\begin{split}
&\frac{d}{dt}\left(\int_\Omega \alpha \chi_2 u\ln u+\beta \chi_1 v\ln v+\frac{\chi_1\chi_2}{2}\frac{|\nabla w|^2}{w}\right)+\alpha \chi_2\int_\Omega \frac{|\nabla u|^2}{u}\\
&\ \ \ + \beta\chi_1\int_\Omega \frac{|\nabla v|^2}{v} +\frac{\chi_1\chi_2}{2}\int_\Omega \frac{\left(\alpha u+\beta v\right)}{w}|\nabla w|^2+\chi_1\chi_2\int_\Omega w|D^2\ln w|^2\\
&=\frac{\chi_1\chi_2 }{2}\int_{\partial\Omega} \frac{1}{w} \frac{\partial |\nabla w|^2}{\partial\nu}.
\end{split}
\ee
Thanks to this crucial evolution identity, we then improve the basic estimates in \eqref{basic-est} and \eqref{reg-2.16-0} to  $(L^1, L^1, L^2)$-boundedness of $(u\ln u, v\ln v, \nabla w)$ as follows.
\begin{lemma}\label{ulnu+gradv2-bdd-lem} Under the basic conditions in Lemma \ref{local-in-time}, for any $\tau\in(0, T_m)$, there exists $C=C(u_0,v_0, w_0, \tau, |\Omega|)>0$  such that the solution $(u,v, w)$ of \eqref{PPT} verifies
\be\label{ulnu+gradv-int}
\int_\Omega \left(|u\ln u|+|v\ln v|+  |\nabla w|^2\right) (\cdot, t)  \leq C, \  \ \    \forall t\in (\tau, T_m).
\ee
\end{lemma}
\begin{proof} In the case that $\Omega$ is convex (as  the case in \cite{TW12-JDE}), notice that    $\frac{\partial}{\partial\nu}|\nabla w|^2\leq 0$ due to $\frac{\partial w}{\partial \nu}=0$, and so the  boundary integral on the right-hand side of \eqref{ulnu+gradv/v-fin} is non-positive. Then a simple integration of \eqref{ulnu+gradv/v-fin}  from $\tau$ to $t$ gives rise to
\be\label{ulnu-vlnv-est1}
\begin{split}
&\int_\Omega \left(\alpha \chi_2 u\ln u+\beta \chi_1 v\ln v+\frac{\chi_1\chi_2}{2}\frac{|\nabla w|^2}{w}\right)(\cdot, t) \\
&\leq \int_\Omega \left(\alpha \chi_2 u\ln u+\beta \chi_1 v\ln v+\frac{\chi_1\chi_2}{2}\frac{|\nabla w|^2}{w}\right)(\cdot, \tau),  \ \ \ \forall t\in (\tau, T_m).
\end{split}
\ee
On the other hand, the algebraic fact that $ -z\ln z\leq e^{-1}$ for all $z>0$ implies
\be\label{abs} \begin{split}
&\int_\Omega \left(\alpha \chi_2 |u\ln u|+\beta \chi_1 |v\ln v|\right)(\cdot, t)\\
&=\alpha \chi_2\left( \int_\Omega u\ln u-2\int_{\{0<u<1\}}u\ln u\right)\\
&\ \ +\beta \chi_1\left( \int_\Omega v\ln v-2\int_{\{0<v<1\}}v\ln v\right) \\
&\leq \int_\Omega \left(\alpha \chi_2 u\ln u+\beta \chi_1 v\ln v\right)(\cdot, t)+2\left(\alpha \chi_2 +\beta \chi_1 \right)e^{-1}|\Omega|,
\end{split}
\ee
which along with \eqref{ulnu-vlnv-est1} and the fact $w\leq \|w_0\|_{L^\infty(\Omega)}$ immediately yields \eqref{ulnu+gradv-int}.

In the general case that $\Omega$ is non-convex,  the idea used to control the boundary integral in \eqref{ulnu+gradv/v-fin} has been detailed in \cite[(3.24)-(3.28)]{Xiang18-NA}. Here, we provide another version which does not involve $L^\infty$-norm of $w$ and thus somehow refine the outcome in \cite[(3.28)]{Xiang18-NA}. Indeed, thanks to \cite[Lemma 3.3]{Win12-CPDE}, we have
\be\label{d2v4-vd2lnv}
\int_\Omega \frac{|\nabla w|^4}{w^3}\leq (2+\sqrt{n})^2\int_\Omega w|D^2\ln w|^2,
\ee
which  together with \eqref{vd2-v3} further shows
\be\label{d2v-vd2lnv}
\int_\Omega \frac{|D^2 w|^2}{w}\leq 2[(2+\sqrt{n})^2+1]\int_\Omega |D^2\ln w|^2.
\ee
Now, repeating the arguments in \cite[(3.26)-(3.28)]{Xiang18-NA}, using \eqref{d2v4-vd2lnv} and  \eqref{d2v-vd2lnv}, we conclude,  for any $\epsilon>0$, there exists $C_\epsilon>0$ such that
 \be\label{boundary-esti-fin}
\begin{split}
&\frac{\chi_1\chi_2 }{2}\int_{\partial\Omega} \frac{1}{w} \frac{\partial |\nabla w|^2}{\partial\nu}\\
&\leq \frac{\epsilon}{2} \int_\Omega \left(\frac{2|D^2w|^2}{w}+\frac{|\nabla w|^4}{2w^3}\right)+C_\epsilon \int_\Omega\frac{|\nabla w|^2}{w}\\
&\leq  \epsilon  \int_\Omega \left(\frac{ |D^2w|^2}{w}+\frac{|\nabla w|^4}{w^3}\right)+C_\epsilon\int_\Omega w\\
&\leq \epsilon \left\{2\left[(2+\sqrt{n})^2+1\right]+(2+\sqrt{n})^2\right\} \int_\Omega w|D^2\ln w|^2+C_\epsilon \int_\Omega w_0,
\end{split} \ee
where, from the first  to the second inequality, we have used H\"{o}lder's inequality to estimate: for any $\eta>0$, there exists $C_\eta>0$ such that
\be\label{wgrad-est}
\int_\Omega\frac{|\nabla w|^2}{w}\leq \eta \int_\Omega\frac{|\nabla w|^4}{w^3}+C_\eta \int_\Omega w.
\ee
Now, choosing
$$
 \epsilon=\min\left\{1, \ \ \  \frac{\frac{\chi_1\chi_2}{2}}{2[(2+\sqrt{n})^2+1]+(2+\sqrt{n})^2}\right\}
$$
 in \eqref{boundary-esti-fin}, and then, substituting  \eqref{boundary-esti-fin} into \eqref{ulnu+gradv/v-fin}, we obtain a key  ordinary differential inequality (ODI) as follows:
 \be\label{ulnu+gradv/v-fin+}
\begin{split}
&\frac{d}{dt}\left(\int_\Omega \alpha \chi_2 u\ln u+\beta \chi_1 v\ln v+\frac{\chi_1\chi_2}{2}\frac{|\nabla w|^2}{w}\right)+\alpha \chi_2\int_\Omega \frac{|\nabla u|^2}{u}\\
&\ \ \ + \beta\chi_1\int_\Omega \frac{|\nabla v|^2}{v} +\frac{\chi_1\chi_2}{2}\int_\Omega w|D^2\ln w|^2\leq C_1.
\end{split}
\ee
By the  Gagliardo-Nirenberg  inequality (cf. Lemma \ref{GN-inter}) and the mass conservations of $u$ and $v$ in  \eqref{basic-est},  one can easily show (cf. \cite[(3.36)]{Xiang18-NA}, for any $\eta>0$,
\be\label{ulnu-GN}
\int_\Omega  u\ln u \leq  \eta\int_\Omega \frac{|\nabla u|^2}{u}+C_\eta, \ \ \ \ \int_\Omega  v\ln v \leq  \eta\int_\Omega \frac{|\nabla v|^2}{v}+C_\eta.
\ee
On the other hand, by \eqref{d2v4-vd2lnv} and \eqref{wgrad-est}, we readily infer that
\be\label{wgrad-est+}
\int_\Omega\frac{|\nabla w|^2}{w}\leq \eta \int_\Omega w|D^2\ln w|^2 +C_\eta.
\ee

Combining \eqref{ulnu-GN} and \eqref{wgrad-est+}  with \eqref{ulnu+gradv/v-fin+} and choosing sufficiently small $\eta>0$, we establish a key final ODI of the form:
\be\label{ulnu+gradv/v-odi}
\begin{split}
&\frac{d}{dt}\left(\int_\Omega \alpha \chi_2 u\ln u+\beta \chi_1 v\ln v+\frac{\chi_1\chi_2}{2}\frac{|\nabla w|^2}{w}\right)\\
&+C_2\left(\int_\Omega \alpha \chi_2 u\ln u+\beta \chi_1 v\ln v+\frac{\chi_1\chi_2}{2}\frac{|\nabla w|^2}{w}\right)\leq C_3.
\end{split}
\ee
Solving the  Gronwall inequality \eqref{ulnu+gradv/v-odi},  using  the trick as in \eqref{abs} and noticing the fact $w\leq \|w_0\|_{L^\infty(\Omega)}$, we finally end up with the desired estimate \eqref{ulnu+gradv-int}.
\end{proof}
In signal production single species chemotaxis models,  the boundedness information provided in \eqref{ulnu+gradv-int} is quite known to allow one to infer 2D global boundedness, cf. \cite{BBTW15, Xiangjde, Xiang18-JMP}. Here, in signal consumption multi-species cases, instead of using the technique used  in  \cite{TW12-JDE}, we shall also show that the compound boundedness in  \eqref{ulnu+gradv-int} enable us to derive first the  $(L^2, L^2, L^4)$-boundedness of $(u,v,\nabla w)$ and then  $(L^\infty, L^\infty, W^{1,\infty})$-boundedness of $(u,v,w)$, which  clarifies the boundedness proof for the case of $a_2\leq0$ in  \cite[the right-hand side of (3.34) therein indeed depends on $t$]{Xiang18-NA}.

\begin{lemma}\label{uvl2-gradwl4-bdd-lem} Let $\Omega\subset \mathbb{R}^2$ be bounded and smooth. Then for any $\tau\in(0, T_m)$, there exists $C=C(u_0,v_0, w_0, \tau, |\Omega|)>0$  such that the solution $(u,v, w)$ of \eqref{PPT} verifies
\be\label{uvl2+gradw-int}
\int_\Omega \left(u^2+v^2+  |\nabla w|^4\right) (\cdot, t)  \leq C, \  \ \    \forall t\in (\tau, T_m).
\ee
\end{lemma}
\begin{proof} Integrating by parts from \eqref{PPT} and using Cauchy-Schwarz inequality, we get
\be\label{uvl2-est}
\begin{cases}
\frac{d}{dt}\int_\Omega u^2+\int_\Omega  |\nabla u|^2\leq\chi_1^2\int_\Omega u^2|\nabla w|^2,\ \ \ \  \forall t\in (0, T_m), \\[0.25cm]
\frac{d}{dt}\int_\Omega v^2+\int_\Omega  |\nabla v|^2\leq\chi_2^2\int_\Omega v^2|\nabla w|^2, \ \ \ \ \forall t\in (0, T_m).
\end{cases}
\ee
Similarly, taking  gradient of the $w$-equation and then multiplying  it by $|\nabla w|^2\nabla w$  and finally integrating over $\Omega$ by parts, we derive that
\be\label{wl4-id}
\begin{split}
&\frac{d}{dt}\int_\Omega |\nabla w|^4+2\int_\Omega |\nabla |\nabla w|^2|^2+4\int_\Omega |\nabla w|^2 |D^2w|^2\\
&=2\int_{\partial\Omega} |\nabla w|^2\frac{\partial}{\partial \nu} |\nabla w |^2+4\int_\Omega (\alpha u+\beta v)w \nabla|\nabla w|^2\nabla w\\
&\ \ +4\int_\Omega (\alpha u+\beta v)w  |\nabla w|^2\Delta w.
\end{split}
\ee
Next, based on \eqref{uvl2-est} and \eqref{wl4-id}, we estimate the terms on the right-hand side of \eqref{wl4-id}. For the boundary integral,  one can  use  (cf. \cite{TX18-JMAA, Xiang18-SIAM, Xiang18-JMP}) the boundary trace embedding  to bound it in terms of the boundedness   of  $\|\nabla  w \|_{L^2}^2$ in \eqref{ulnu+gradv-int} to infer, for any $\epsilon>0$,  there exists $C_\epsilon>0$ such that
\be\label{com-est-4}
\begin{split}
2\int_{\partial\Omega}  |\nabla w|^{2}\frac{\partial}{\partial \nu} |\nabla w|^2&\leq \epsilon \int_\Omega |\nabla |\nabla w|^2|^2+C_\epsilon \left(\int_{\Omega} |\nabla w|^2\right)^2\\
&\leq \epsilon \int_\Omega |\nabla |\nabla w|^2|^2+C_\epsilon,\ \ \   \forall \epsilon>0.
\end{split}
\ee
Noticing $\|w\|_{L^\infty}\leq \|w_0\|_{L^\infty}$, we deduce that
\be\label{mix-est1}
\begin{split}
&4\int_\Omega \left(\alpha u+\beta v\right)w \nabla|\nabla w|^2\nabla w\\
&\leq  \frac{1}{2}\int_\Omega |\nabla|\nabla w|^2|^2+8\|w_0\|_{L^\infty}^2 \int_\Omega \left(\alpha u+\beta v\right)^2|\nabla w|^2 \\
&\leq \frac{1}{2}\int_\Omega |\nabla|\nabla w|^2|^2+16\alpha^2\|w_0\|_{L^\infty}^2 \int_\Omega u^2|\nabla w|^2 +16\beta^2\|w_0\|_{L^\infty}^2 \int_\Omega v^2|\nabla w|^2
\end{split}
\ee
and, similarly, since $|\Delta w|^2\leq 2|D^2w|^2$, we obtain that
\be\label{mix-est2}
\begin{split}
&4\int_\Omega (\alpha u+\beta v)w  |\nabla w|^2\Delta w\\
&\leq 4\int_\Omega |\nabla w|^2 |D^2w|^2+4\alpha^2\|w_0\|_{L^\infty}^2 \int_\Omega u^2|\nabla w|^2 +4\beta^2\|w_0\|_{L^\infty}^2 \int_\Omega v^2|\nabla w|^2.
\end{split}
\ee
Combining   the estimates \eqref{uvl2-est}, \eqref{wl4-id}, \eqref{com-est-4} with $\epsilon=\frac{1}{2}$,  \eqref{mix-est1} and \eqref{mix-est2}, we conclude a Key ODI as follows, for $ t\in (0, T_m)$,
\be\label{uvl2-gradwl4-est}
\begin{split}
&\frac{d}{dt}\int_\Omega \left(u^2+v^2+|\nabla w|^4\right)+\int_\Omega \left( |\nabla u|^2+|\nabla v|^2+|\nabla |\nabla w|^2|^2\right)\\
&\leq\left(\chi_1^2+20\alpha^2\|w_0\|_{L^\infty}^2\right) \int_\Omega u^2|\nabla w|^2+\left(\chi_2^2+20\beta^2\|w_0\|_{L^\infty}^2\right) \int_\Omega v^2|\nabla w|^2.
\end{split}
\ee
Now, the Young's inequality with epsilon shows, for any $\epsilon_1>0$, that
\be\label{yound-epi}
\int_\Omega u^2|\nabla w|^2+ \int_\Omega v^2|\nabla w|^2\leq 2\epsilon_1 \int_\Omega |\nabla w|^6+\frac{2}{3\sqrt{3\epsilon_1}} \int_\Omega u^3+\frac{2}{3\sqrt{3\epsilon_1}} \int_\Omega v^3.
\ee
Now, thanks to the compound boundedness information in \eqref{ulnu+gradv-int}, using the usual 2D GN inequality as in Lemma \ref{GN-inter} and its extended version involving logarithmic functions  (cf. \cite[Lemma A. 5]{TW14-JDE} or \cite[Lemma 3.4]{XZ19-non}), we can easily deduce, for any $\epsilon_2, \epsilon_3$,  there exist $C_{\epsilon_2}, C_{\epsilon_3}>0$ and $C>0$ such that
\be\label{uvwl3-l6-est}
\begin{cases}
 \int_\Omega u^2+ \int_\Omega u^3\leq\epsilon_2\int_\Omega  |\nabla u|^2+C_{\epsilon_2}, \\[0.25cm]
\int_\Omega v^2+ \int_\Omega v^3\leq\epsilon_3\int_\Omega  |\nabla v|^2+C_{\epsilon_3}, \\[0.25cm]
\int_\Omega  |\nabla w|^6\leq C\int_\Omega |\nabla |\nabla w|^2|^2+C.
\end{cases}
\ee
Substituting  \eqref{yound-epi} and \eqref{uvwl3-l6-est} into \eqref{uvl2-gradwl4-est} and then choosing sufficiently small $\epsilon_i>0$, we finally obtain a simple yet important ODE as follows:
$$
 \frac{d}{dt}\int_\Omega \left(u^2+v^2+|\nabla w|^4\right)+\int_\Omega \left(u^2+v^2+|\nabla w|^4\right)\leq C,
$$
which immediately entails \eqref{uvl2+gradw-int}.
\end{proof}
\begin{proof}[\bf{Proof of 2D global existence, boundedness and convergence}] In light of  the gained  $(L^2, L^2, L^4)$-boundedness of $(u,v,\nabla w)$ in \eqref{uvl2+gradw-int} and the $L^\infty$-boundedness of $w$ in \eqref{reg-2.17-0}, using semigroup type arguments to $w$-equation, one can easily derive first $ W^{1, q}$-boundedness of $w$ for any finite $q$, and then, testing the $u,v$-equations to derive $(L^3, L^3)$-boundedness of $(u,v)$, and then,  using semigroup type arguments to $w$-equation again to derive $ W^{1, \infty}$-boundedness of $w$, and finally, applying semigroup type arguments to $u,v$-equations to derive   $(L^\infty, L^\infty, W^{1, \infty})$-boundedness of $(u,v,  w)$ as in \eqref{linfty-bdd}, see details in    e.g. \cite{BBTW15, XZ19-non, ZT19}, for instance. The convergence in  \eqref{uvw-cov-zt} goes in the same way as \cite{TW12-JDE, ZT19}.
\end{proof}

\section{Preliminaries on weak solutions in 3D or higher dimensions}
In this section, we first introduce the concept of weak solutions, and then, we state some useful lemmas for later use.
\begin{definition}\label{weak-uvw} By a global weak solution of \eqref{PPT}, we mean a triple $(u,v,w)$ of nonnegative functions
\begin{equation*}
n=3:\ \ \begin{cases}
u\in L^1_{loc}([0,\infty);L^1(\Omega)),\\
v\in L^1_{loc}([0,\infty);L^1(\Omega))\mbox{ and }\\
w\in L^1_{loc}([0,\infty);W^{1,1}(\Omega)),
\end{cases}  \end{equation*}
\begin{equation*} n=4,5:\ \ \begin{cases}
u\in L^1_{loc}([0,\infty);W^{1,1}(\Omega)),\\
v\in L^1_{loc}([0,\infty);W^{1,1}(\Omega))\mbox{ and }\\
w\in L^{\infty}_{loc}(\overline{\Omega}\times[0,\infty))\cap L^1_{loc}([0,\infty);W^{1,1}(\Omega)),
\end{cases}  \end{equation*}
such that
\begin{equation*}
uw,\ vw,\ u\nabla w\ and \ v\nabla w\  belong\  to \ L^1_{loc}([0,\infty);L^1(\Omega))
\end{equation*}
and that the following identities motivated from integration by parts
\be\label{reg-2.4}
-\int^{\infty}_0\int_{\Omega}u\varphi_t-\int_{\Omega}u_0\varphi(\cdot,0)=\int^{\infty}_0\int_{\Omega}\nabla u\cdot\nabla\varphi+\chi_1\int^{\infty}_0\int_{\Omega}u\nabla w\cdot\nabla\varphi,
\ee
\be\label{reg-2.5}
-\int^{\infty}_0\int_{\Omega}v\varphi_t-\int_{\Omega}v_0\varphi(\cdot,0)
=-\int^{\infty}_0\int_{\Omega}\nabla v\cdot\nabla\varphi+\chi_2\int^{\infty}_0\int_{\Omega}v\nabla w\cdot\nabla\varphi
\ee
and
\be\label{reg-2.6}
-\int^{\infty}_0\int_{\Omega}w\varphi_t-\int_{\Omega}w_0\varphi(\cdot,0)
=-\int^{\infty}_0\int_{\Omega}\nabla w\cdot\nabla\varphi-\int^{\infty}_0\int_{\Omega}(\alpha u+\beta v)w\varphi
\ee
hold for all $\varphi\in C^{\infty}_0(\Omega\times[0,\infty))$.
\end{definition}

In order to achieve global solvability within this framework through an appropriate regularization process, for $\varepsilon\in(0,1)$, let us define $ F_\varepsilon: [0, \infty)\mapsto \mathbb{R}^+$ by
\be\label{reg-2.7}
 F_\varepsilon(s):=\begin{cases}
 \frac{1}{\varepsilon}\ln(1+\varepsilon s),    &\text{if  } n=3, \\[0.2cm]
 \frac{s}{1+\varepsilon s},    &\text{if  } n=4, 5.
 \end{cases}
\ee
It then follows easily that the $C^{\infty}([0,\infty))$-family $(F_\varepsilon)_{\varepsilon\in(0,1)}$ has the properties that
\be\label{reg-2.9}
 F_\varepsilon(0)=0,  \ \  F_\varepsilon(s)\rightarrow s \text{ as } \varepsilon\searrow 0  \ \mbox{ and } \ \ 0< F'_\varepsilon(s)\leq1 \ \ \mbox{ for all} \ s\geq0,
\ee
and that, for any $s\geq0$,
\be\label{reg-2.10}
  0\leq F'_\varepsilon(s)\nearrow1 \mbox{ as } \varepsilon\searrow0, \ \ \ \  0\leq sF'_\varepsilon(s) \leq \frac{1}{\varepsilon}, \  \  0\leq -sF''_\varepsilon(s) \leq 2,   \ \forall \varepsilon\in(0,1).
\ee
Then, for $\varepsilon\in (0, 1)$, we consider the following regularized problem:
\be \label{reg-2.11}\begin{cases}
u_{\varepsilon t} = \Delta u_{\varepsilon}-\chi_1 \nabla \cdot (u_{\varepsilon}F'_{\varepsilon}(u_{\varepsilon})\nabla w_{\varepsilon}),  &x\in \Omega, t>0, \\[0.2cm]
v_{\varepsilon t} = \Delta v_{\varepsilon}-\chi_2 \nabla \cdot (v_{\varepsilon}F'_{\varepsilon}(v_{\varepsilon})\nabla w_{\varepsilon}), & x\in \Omega, t>0, \\[0.2cm]
w_{\varepsilon t} = \Delta w_{\varepsilon} -(\alpha F_{\varepsilon}(u_{\varepsilon})+\beta F_{\varepsilon}(v_{\varepsilon}))w_{\varepsilon},  & x\in \Omega, t>0, \\[0.2cm]
\frac{\partial u_{\varepsilon}}{\partial \nu}=\frac{\partial v_{\varepsilon}}{\partial \nu}=\frac{\partial w_{\varepsilon}}{\partial \nu}=0, & x\in \partial \Omega, t>0,\\[0.2cm]
u_{\varepsilon}(x,0)=u_0(x) \ , v_{\varepsilon}(x,0)=v_0(x) \ ,  w_{\varepsilon}(x,0)=w_0(x) \ , & x\in \Omega.  \end{cases}  \ee

The framework of contraction mapping argument first allows one to conclude local well-posedness on $(0, T_{m, \varepsilon})$, and then, given the basic estimates in  Lemmas \ref{lemma-2.4} and \ref{lemma-2.5} on  $(0, T_{m, \varepsilon})$ and the choice of $F_\varepsilon$ in \eqref{reg-2.7}, using Neumann semigroup estimates to the $w$-equation in \eqref{reg-2.11},  one can easily see that $\|\nabla w_\varepsilon\|_{L^\infty}$ is uniformly bounded on $(0, T_{m, \varepsilon})$, and this allows one to conclude finally $T_{m, \varepsilon}=\infty$, namely, the global existence of classical solution to the approximating system \eqref{reg-2.11}; see,  quite  detailed display of  similar reasonings in related circumstances \cite{Hor-Win, Tao11, TW12-JDE, Win12-CPDE, Win17-JDE}.
\begin{lemma} Let  $\chi_1,\chi_2>0$ and $\alpha,\beta>0$ and $\Omega\subset\mathbb{R}^n(n\geq 1)$ be a bounded and smooth domain, and, let $F_\varepsilon$ be defined by \eqref{reg-2.7} and initial data $(u_0,v_0,w_0)$ satisfy \eqref{initial-data-reg}. Then for each $\varepsilon\in(0,1)$, the system \eqref{reg-2.11} admits a global classical solution $(u_{\varepsilon},v_{\varepsilon}, w_{\varepsilon})$ such that $u_{\varepsilon}>0$, $v_{\varepsilon}>0$ and $w_{\varepsilon}>0$ on  $\bar{\Omega}\times(0,\infty)$.
\end{lemma}

\begin{lemma}\label{lemma-2.4} For all $\varepsilon\in(0,1)$, the solution of \eqref{reg-2.11} satisfies, for $t>0$,
\be\label{reg-2.15}
\|u_{\varepsilon}(\cdot,t)\|_{L^1(\Omega)}=\|u_0\|_{L^1(\Omega)}, \ \ \|v_{\varepsilon}(\cdot,t)\|_{L^1(\Omega)}=\|v_0\|_{L^1(\Omega)}.
\ee
\end{lemma}
\begin{proof}
Integrating the first and second equations in \eqref{reg-2.11} and the no-flux boundary conditions, we immediately obtain \eqref{reg-2.15}.
\end{proof}

\begin{lemma}\label{lemma-2.5} Let $\varepsilon\in(0,1)$ and $p\in[1,\infty]$. Then  the solution of \eqref{reg-2.11} verifies
\be\label{reg-2.16}
t\mapsto\|w_{\varepsilon}(\cdot,t)\|_{L^p(\Omega)} \ \ is \ nonincreasing \ in \ [0,\infty).
\ee
In particular,
\be\label{reg-2.17}
\|w_{\varepsilon}(\cdot,t)\|_{L^p(\Omega)}\leq\|w_0\|_{L^p(\Omega)}.
\ee
\end{lemma}
\begin{proof} Notice that  $F_\varepsilon$ and $\alpha, \beta, u_\varepsilon,v_\varepsilon, w_\varepsilon$ are nonnegative, we have $w_{\varepsilon t}\leq \Delta w_{\varepsilon}$, and thus, using the maximum principle and energy estimate, we readily derive \eqref{reg-2.16} and   \eqref{reg-2.17}, see details in  Lemma \ref{local-in-time}.
\end{proof}

\begin{lemma} For all $\varepsilon\in(0,1)$. we have
\be\label{reg-2.18}
\alpha\int^{\infty}_0\int_{\Omega}F_\varepsilon(u_{\varepsilon})w_{\varepsilon}+\beta\int^{\infty}_0\int_{\Omega}F_\varepsilon(v_{\varepsilon})w_{\varepsilon}\leq\int_{\Omega}w_0.
\ee
In particular, the limit triple $(u,v,w)$ defined by Lemma \ref{lemma-4.8} satisfies
\be\label{reg-2.19}
\alpha\int^{\infty}_0\int_{\Omega}u w+\beta\int^{\infty}_0\int_{\Omega}v w\leq\int_{\Omega}w_0.
\ee
\end{lemma}
\begin{proof} Integrating the third equation in \eqref{reg-2.11} we get
$$
\int_{\Omega}w_{\varepsilon}(\cdot,t)+\alpha\int^t_0
\int_{\Omega}F_\varepsilon(u_{\varepsilon})w_{\varepsilon}
+\beta\int^t_0\int_{\Omega}F_\varepsilon(v_{\varepsilon})w_{\varepsilon}=\int_{\Omega}w_0
$$
for all $t\geq0$. Due to $w_{\varepsilon}\geq0$, it immediately derives \eqref{reg-2.18}. Applying the Fatou's lemma, we also readily conclude \eqref{reg-2.19}. This completes the proof.
\end{proof}

\section{Global dynamics of weak solutions  in 3,4 and 5D}
In this section, we first establish important a-priori $\varepsilon$-independent estimates for classical solutions to  \eqref{reg-2.11} and then we pass to the limit as $\varepsilon\rightarrow 0$ to show  the global existence and convergence of weak solutions in the sense of Definition \ref{weak-uvw} for the IBVP \eqref{PPT} in $3,4,5$-D, and thus accomplishing (B2) and (B3) in Theorem \ref{global-ext-3d-4d-5d}.
\begin{lemma}\label{lemma-4.1}  There exists positive constant  $K_1=K_1(\|u_0\|_{L^1},\|v_0\|_{L^1}, \|w_0\|_{L^1})>0$ such that the global solution of \eqref{reg-2.11} fulfills,  for  all $\varepsilon\in(0,1)$ and $t>0$,
\be\label{int-4.30}\begin{split}
 & \frac{d}{dt} \int_\Omega \left(\alpha \chi_2u_{\varepsilon} \ln u_{\varepsilon}+ \beta \chi_1  v_{\varepsilon}\ln v_{\varepsilon}+\frac{\chi_1\chi_2}{2} \frac{|\nabla w_{\varepsilon}|^2} {w_{\varepsilon}} \right)\\
&\ \  +\int_\Omega \left(\alpha \chi_2u_{\varepsilon} \ln u_{\varepsilon}+ \beta \chi_1  v_{\varepsilon}\ln v_{\varepsilon}+\frac{\chi_1\chi_2}{2} \frac{|\nabla w_{\varepsilon}|^2} {w_{\varepsilon}} \right)\\
&\ \  +\frac{\alpha\chi_2}{2}\int_\Omega\frac{|\nabla u_{\varepsilon}|^2}{u_{\varepsilon}} + \frac{\beta \chi_1}{2}\int_\Omega\frac{|\nabla v_{\varepsilon}|^2}{v_{\varepsilon}}+\frac{\chi_1\chi_2}{2}\int_\Omega w_{\varepsilon}|D^2\ln w_{\varepsilon}|^2\\
&\  \  +\frac{\chi_1\chi_2}{2}\int_\Omega \left(\alpha F_\varepsilon(u_{\varepsilon})+\beta F_\varepsilon(v_{\varepsilon})\right)\frac{|\nabla w_{\varepsilon}|^2}{w_{\varepsilon}}\\
&\leq \begin{cases} 0, & \text{if } \Omega  \text{ is convex}, \\[0.2cm]
 K_1, & \text{if } \Omega  \text{ is non-convex}.
 \end{cases}
\end{split}
\ee
Hence, there exists $K_2:=K_2(u_0, v_0, w_0)>0$ such that
\be\label{ulnu+gradv-int-eps}
\int_\Omega \left(|u_\varepsilon\ln u_\varepsilon|+|v_\varepsilon\ln v_\varepsilon|+  |\nabla w
_\varepsilon|^2+\frac{|\nabla w
_\varepsilon|^2}{w
_\varepsilon}\right) (\cdot, t)  \leq K_2, \  \ \    \forall t\in (0,\infty).
\ee
\end{lemma}
\begin{proof} Conducting similar computations leading to \eqref{ulnu+gradv/v-fin}, we calculate that
\be\label{int-4.3}\begin{split}
 & \frac{d}{dt} \int_\Omega \left(\alpha \chi_2u_{\varepsilon} \ln u_{\varepsilon}+ \beta \chi_1  v_{\varepsilon}\ln v_{\varepsilon}+\frac{\chi_1\chi_2}{2} \frac{|\nabla w_{\varepsilon}|^2} {w_{\varepsilon}} \right)\\
&\ \  +\int_\Omega \left(\alpha \chi_2u_{\varepsilon} \ln u_{\varepsilon}+ \beta \chi_1  v_{\varepsilon}\ln v_{\varepsilon}+\frac{\chi_1\chi_2}{2} \frac{|\nabla w_{\varepsilon}|^2} {w_{\varepsilon}} \right)\\
&\ \  +\alpha\chi_2\int_\Omega\frac{|\nabla u_{\varepsilon}|^2}{u_{\varepsilon}} + \beta \chi_1\int_\Omega\frac{|\nabla v_{\varepsilon}|^2}{v_{\varepsilon}}+\chi_1\chi_2\int_\Omega w_{\varepsilon}|D^2\ln w_{\varepsilon}|^2\\
&\  \  +\frac{\chi_1\chi_2}{2}\int_\Omega \left(\alpha F_\varepsilon(u_{\varepsilon})+\beta F_\varepsilon(v_{\varepsilon})\right)\frac{|\nabla w_{\varepsilon}|^2}{w_{\varepsilon}}\\
&=\int_\Omega \left(\alpha \chi_2u_{\varepsilon} \ln u_{\varepsilon}+ \beta \chi_1  v_{\varepsilon}\ln v_{\varepsilon}+\frac{\chi_1\chi_2}{2} \frac{|\nabla w_{\varepsilon}|^2} {w_{\varepsilon}} \right)+\frac{\chi_1\chi_2 }{2}\int_{\partial\Omega} \frac{1}{w} \frac{\partial |\nabla w|^2}{\partial\nu}.
\end{split}
\ee
By the $L^1$-boundedness of $u_\varepsilon$ and $v_\varepsilon$  in \eqref{reg-2.15}, an straightforward application of \eqref{ulnu-GN} shows that
\be\label{int-4.7}
 \alpha \chi_2 \int_\Omega u_{\varepsilon} \ln u_{\varepsilon}\leq \frac{ \alpha \chi_2}{2}\int_\Omega\frac{|\nabla u_{\varepsilon}|^2}{u_{\varepsilon}}+C_1(\|u_0\|_{L^1})
\ee
and
\be\label{int-4.8}
\beta \chi_1 \int_\Omega v_{\varepsilon}\ln v_{\varepsilon}\leq \frac{\beta \chi_1}{2} \int_\Omega\frac{|\nabla v_{\varepsilon}|^2}{v_{\varepsilon}}+C_2(\|v_0\|_{L^1}).
\ee
Also, thanks to \eqref{reg-2.17},   a simple use of \eqref{wgrad-est+}   entails
\be\label{int-4.10}
\frac{\chi_1\chi_2}{2} \int_\Omega   \frac{|\nabla w_{\varepsilon}|^2} {w_{\varepsilon}} \leq \frac{\chi_1\chi_2}{4}\int_\Omega w_{\varepsilon}|D^2\ln w_{\varepsilon}|^2+C_3(\|w_0\|_{L^1})
\ee
and, in the case that $\Omega$ is non-convex, an easy use of \eqref{boundary-esti-fin} shows that
 \be\label{bddary est}\begin{split}
&\frac{\chi_1\chi_2 }{2}\int_{\partial\Omega} \frac{1}{w_\varepsilon} \frac{\partial |\nabla w_\varepsilon|^2}{\partial\nu} \\
&\leq \begin{cases} 0, & \text{if } \Omega \text{ is convex}, \\[0.2cm]
 \frac{\chi_1\chi_2}{4} \int_\Omega w_{\varepsilon}|D^2\ln w_{\varepsilon}|^2+C_4(\|w_0\|_{L^1}), & \text{if } \Omega \text{ is non-convex}.
 \end{cases}
 \end{split}
 \ee
Substituting \eqref{int-4.7}, \eqref{int-4.8},\eqref{int-4.10} and \eqref{bddary est} into \eqref{int-4.3}, we derive \eqref{int-4.30}. Since $(u_\varepsilon, v_\varepsilon, w_\varepsilon)$ satisfies an ODI of the form \eqref{ulnu+gradv/v-odi}, and so \eqref{ulnu+gradv-int-eps} follows similarly as  \eqref{ulnu+gradv-int}.
\end{proof}
\subsection{$\varepsilon$-independent estimates for the regularized problem}
\begin{lemma}\label{lemma-4.2} There exists $K_3=K_3(u_0, v_0, w_0)>0$ such that the global solution of \eqref{reg-2.11} fulfills,   for  all $\varepsilon\in(0,1)$,
\be\label{int-4.16}
\begin{split}
&\int^\infty_0\int_\Omega\left(\frac{|\nabla u_{\varepsilon}|^2}{u_{\varepsilon}}+\frac{|\nabla v_{\varepsilon}|^2}{v_{\varepsilon}}+|D^2w_{\varepsilon}|^2+|\nabla w_{\varepsilon}|^4\right)\\
& \ \ +\int^\infty_0\int_\Omega\left(F_{\varepsilon}(u_{\varepsilon})|\nabla w_{\varepsilon}|^2+F_{\varepsilon}(v_{\varepsilon})|\nabla w_{\varepsilon}|^2\right)\leq K_3, \ \  \ \text{if } \Omega  \text{ is } convex.
\end{split}
\ee
 There exists $K_4=K_4(u_0, v_0, w_0)>0$ such that the global solution of \eqref{reg-2.11} fulfills, for any $\varepsilon\in(0,1)$ and $t\in[0, \infty)$,
\be\label{st-comp-est}
\begin{split}
&\int^{t+1}_t\int_\Omega\left(\frac{|\nabla u_{\varepsilon}|^2}{u_{\varepsilon}}+\frac{|\nabla v_{\varepsilon}|^2}{v_{\varepsilon}}+|D^2w_{\varepsilon}|^2+|\nabla w_{\varepsilon}|^4\right)\\
& +\int^{t+1}_t\int_\Omega\left(F_{\varepsilon}(u_{\varepsilon})|\nabla w_{\varepsilon}|^2+F_{\varepsilon}(v_{\varepsilon})|\nabla w_{\varepsilon}|^2\right)\leq K_4, \  \text{if } \Omega  \text{ is non-convex}.
\end{split}
\ee
\end{lemma}
\begin{proof} In the case that  $\Omega$ is convex, integrating \eqref{int-4.3} with respect to $t\in(0,\infty)$ and using the fact that $ -z\ln z\leq e^{-1}$ for all $z>0$,  we have
\be\label{int-st-step1}\begin{split}
& \int_0^t \int_\Omega\left( \alpha\chi_2\frac{|\nabla u_{\varepsilon}|^2}{u_{\varepsilon}} + \beta \chi_1 \frac{|\nabla v_{\varepsilon}|^2}{v_{\varepsilon}}+\chi_1\chi_2  w_{\varepsilon}|D^2\ln w_{\varepsilon}|^2\right)\\
&\  +\frac{\chi_1\chi_2}{2}\int_0^t \int_\Omega \left(\alpha F_\varepsilon(u_{\varepsilon})+\beta F_\varepsilon(v_{\varepsilon})\right)\frac{|\nabla w_{\varepsilon}|^2}{w_{\varepsilon}}\\
&\leq \int_\Omega \left(\alpha \chi_2u_0 \ln u_0+ \beta \chi_1  v_0\ln v_0+\frac{\chi_1\chi_2}{2} \frac{|\nabla w_0|^2} {w_0} \right)\\
&\ \  -\int_\Omega \left(\alpha \chi_2u_{\varepsilon} \ln u_{\varepsilon}+ \beta \chi_1  v_{\varepsilon}\ln v_{\varepsilon}\right)\\
&\leq \int_\Omega \left(\alpha \chi_2u_0 \ln u_0+ \beta \chi_1  v_0\ln v_0+\frac{\chi_1\chi_2}{2} \frac{|\nabla w_0|^2} {w_0} \right)+ \left(\alpha \chi_2 +\beta \chi_1\right)e^{-1}|\Omega|.
\end{split}
\ee
In light of \eqref{d2v4-vd2lnv} and \eqref{d2v-vd2lnv} and the fact  $w_{\varepsilon}\leq \|w_0\|_{L^{\infty}}$, we infer that
\be\label{int-4.23}
\frac{1}{\|w_0\|_{L^{\infty}}^3}\int_0^t\int_\Omega |\nabla w_{\varepsilon}|^4\leq \int_0^t\int_\Omega\frac{|\nabla w_{\varepsilon}|^4} {w^3_{\varepsilon}}\leq (2+\sqrt{n})^2\int_0^t\int_\Omega w_{\varepsilon}|D^2\ln w_{\varepsilon}|^2
\ee
and
\be\label{int-4.24}
\frac{1}{\|w_0\|_{L^{\infty}}}\int^t_0\int_\Omega |D^2w_{\varepsilon}|^2\leq2\left[(2+\sqrt{n})^2+1\right]\int^t_0\int_\Omega w_{\varepsilon}|D^2\ln w_{\varepsilon}|^2.
\ee
Combining \eqref{int-4.23} and \eqref{int-4.24} into \eqref{int-st-step1} and sending $t\rightarrow\infty$, we derive \eqref{int-4.16}.

Similarly, when $\Omega$ is nonconvex, substituting  \eqref{bddary est} into \eqref{int-4.3}, we obtain that
\be\label{int-4.3-st-gap1}\begin{split}
 & \frac{d}{dt} \int_\Omega \left(\alpha \chi_2u_{\varepsilon} \ln u_{\varepsilon}+ \beta \chi_1  v_{\varepsilon}\ln v_{\varepsilon}+\frac{\chi_1\chi_2}{2} \frac{|\nabla w_{\varepsilon}|^2} {w_{\varepsilon}} \right)\\
&\ \  +\alpha\chi_2\int_\Omega\frac{|\nabla u_{\varepsilon}|^2}{u_{\varepsilon}} + \beta \chi_1\int_\Omega\frac{|\nabla v_{\varepsilon}|^2}{v_{\varepsilon}}+\frac{3\chi_1\chi_2 }{4}\int_\Omega w_{\varepsilon}|D^2\ln w_{\varepsilon}|^2\\
&\  \  +\frac{\chi_1\chi_2}{2}\int_\Omega \left(\alpha F_\varepsilon(u_{\varepsilon})+\beta F_\varepsilon(v_{\varepsilon})\right)\frac{|\nabla w_{\varepsilon}|^2}{w_{\varepsilon}}\\
&\leq C_4(\|w_0\|_{L^1}).
\end{split}
\ee
Integrating \eqref{int-4.3-st-gap1} from $t$ to $t+1$ and using \eqref{ulnu+gradv-int-eps}, we achieve  \eqref{st-comp-est}.
\end{proof}

\begin{lemma}\label{lemma-4.3} There exists  $K_5=K_5(u_0, v_0, w_0)>0$  such that the global solution of \eqref{reg-2.11} fulfills, for all  $\varepsilon\in(0,1)$ and for any $t\in[0, \infty)$,
\be\label{int-4.29}
\int_t^{t+1}\int_\Omega\left(u_{\varepsilon}^{\frac{n+2}{n}}+v_{\varepsilon}^{\frac{n+2}{n}}+|\nabla u_{\varepsilon}|^{\frac{n+2}{n+1}}+|\nabla v_{\varepsilon}|^{\frac{n+2}{n+1}}\right)\leq K_5.
\ee
\end{lemma}
\begin{proof}
Given the mass conservation of $u_\varepsilon$ in \eqref{reg-2.15},  the Gagliardo-Nirenberg  inequality (cf. Lemma \ref{GN-inter}) allows us to deduce that
\be\label{int-4.41}\begin{split}
 \int_\Omega u_{\varepsilon}^{\frac{n+2}{n}}= \|\sqrt{u_{\varepsilon}}\|^{\frac{2(n+2)}{n}}_{L^\frac{2(n+2)}{n}}&\leq C_1 \|\nabla\sqrt{u_{\varepsilon}}\|^2_{L^2}\|\sqrt{u_{\varepsilon}}\|^\frac{4}{n}_{L^2} +C_1\|\sqrt{u_{\varepsilon}}\|^{\frac{2(n+2)}{n}}_{L^2}\\
&=\frac{C_1}{4}\|u_0\|_{L^1}^\frac{2}{n}\int_\Omega \frac{|\nabla u_{\varepsilon}|^2}{u_{\varepsilon}}+C_1\|u_0\|_{L^1}^\frac{n+2}{n}.
\end{split}
\ee
Likewise, one also that
\be\label{int-4.42}
\int_\Omega v_{\varepsilon}^{\frac{n+2}{n}}\leq C_2\|v_0\|_{L^1}^\frac{2}{n}\int_\Omega \frac{|\nabla v_{\varepsilon}|^2}{v_{\varepsilon}}+C_2\|v_0\|_{L^1}^\frac{n+2}{n}.
\ee
Applying the Young inequality, we obtain
\be\label{int-4.43}
 \int_\Omega |\nabla u_{\varepsilon}|^{\frac{n+2}{n+1}}= \int_\Omega\left(\frac{|\nabla u_{\varepsilon}|^2}{u_{\varepsilon}}\right)^{\frac{n+2}{2(n+1)}} \cdot u_{\varepsilon}^{\frac{n+2}{2(n+1)}}\leq  \int_\Omega\frac{|\nabla u_{\varepsilon}|^2}{u_{\varepsilon}}+ \int_\Omega u_{\varepsilon}^\frac{n+2}{n}.
\ee
Similarly,
\be\label{int-4.44}
\int_\Omega |\nabla v_{\varepsilon}|^{\frac{n+2}{n+1}}\leq  \int_\Omega\frac{|\nabla v_{\varepsilon}|^2}{v_{\varepsilon}}+ \int_\Omega v_{\varepsilon}^\frac{n+2}{n}.
\ee
For any $t\geq 0$, integrating \eqref{int-4.41}, \eqref{int-4.42}, \eqref{int-4.43} and \eqref{int-4.44} from $t$ to $t+1$, then   using \eqref{int-4.16} or \eqref{st-comp-est}, we readily conclude \eqref{int-4.29}.
\end{proof}

\begin{lemma}\label{lemma-4.4} There exists  $K_6=K_6(u_0, v_0, w_0)>0$ such that for all $\varepsilon\in(0,1)$, the global solution of \eqref{reg-2.11} fulfills
\be\label{int-4.45}
\int^\infty_0\left(\|u_{\varepsilon}-\bar{u}_0\|^2_{L^\frac{n}{n-1}}
+\|v_{\varepsilon}-\bar{v}_0\|^2_{L^\frac{n}{n-1}}\right)\leq K_6, \ \   \text{if } \Omega  \text{ is convex};
\ee
and, there exists  $K_7=K_7(u_0, v_0, w_0)>0$ such that, for any  $\varepsilon\in(0,1)$ and $t\geq 0$,
\be\label{int-4.46}
\int_t^{t+1}\left(\|u_{\varepsilon}-\bar{u}_0\|^2_{L^\frac{n}{n-1}}
+|v_{\varepsilon}-\bar{v}_0\|^2_{L^\frac{n}{n-1}}\right)\leq K_7, \ \   \text{if } \Omega  \text{ is non-convex}.
\ee
\end{lemma}
\begin{proof}
The Cauchy-Schwarz inequality entails that
$$
\left(\int_\Omega|\nabla u_{\varepsilon}|\right)^2+\left(\int_\Omega|\nabla v_{\varepsilon}|\right)^2\leq\|u_0\|_{L^1}\int_\Omega\frac{|\nabla u_{\varepsilon}|^2}{u_{\varepsilon}}+\|v_0\|_{L^1}\int_\Omega\frac{|\nabla v_{\varepsilon}|^2}{v_{\varepsilon}}.
$$
 Notice from \eqref{reg-2.15} that $\bar{u}_\varepsilon=\bar{u}_0$ and $\bar{v}_\varepsilon=\bar{v}_0$; we then  we infer from the Sobolev embedding $W^{1,1}(\Omega)\hookrightarrow L^\frac{n}{n-1}(\Omega)$ and  the Poincare inequality  that
 \begin{equation*}
 \begin{split}
\left(\|u_{\varepsilon}-\bar{u}_0\|^2_{L^\frac{n}{n-1}}+\|v_{\varepsilon}
-\bar{v}_0\|^2_{L^\frac{n}{n-1}}\right)\leq C\left(\int_\Omega|\nabla u_{\varepsilon}|\right)^2+C\left(\int_\Omega|\nabla v_{\varepsilon}|\right)^2.
\end{split}
\end{equation*}
Integrating those inequalities from $0$ to $t$ (if $\Omega$ is convex) or from $t$ and $t+1$ (if $\Omega$ is non-convex) and making use of \eqref{int-4.16} or \eqref{st-comp-est}, we readily infer \eqref{int-4.45} and \eqref{int-4.46}.
\end{proof}

In the sequel,  for our subsequent compactness argument, we study  the space-time regularity of the time derivatives  of solutions to the regularized  system \eqref{reg-2.11}.

\begin{lemma}\label{lemma-4.5} There exists  $K_8=K_8(u_0, v_0, w_0)>0$ such that for all $\varepsilon\in(0,1)$, the global solution of \eqref{reg-2.11} fulfills
\be\label{int-4.49}
\int^\infty_0 \int_\Omega w^2_{\varepsilon t}\leq K_8,\ \ \   \text{if  } \Omega  \text{ is convex};
\ee
and, there exists  $K_9=K_9(u_0, v_0, w_0)>0$ such that, for any  $\varepsilon\in(0,1)$ and $t\geq0$,
\be\label{int-4.490}
\int_t^{t+1} \int_\Omega w^2_{\varepsilon t}\leq K_9, \ \ \   \text{if  } \Omega  \text{ is non-convex}.
\ee
\end{lemma}
\begin{proof} Multiplying the third equation in \eqref{reg-2.11} by $2w_{\varepsilon t}$ and then integrating over $\Omega$ by parts, we obtain that
\begin{equation*}\begin{split}
&2\int_\Omega w^2_{\varepsilon t}+\frac{d}{dt}\int_\Omega|\nabla w_{\varepsilon}|^2\\
&\ =- \alpha \int_\Omega F_{\varepsilon}(u_{\varepsilon})(w_{\varepsilon}^2)_t- \beta \int_\Omega F_{\varepsilon}(v_{\varepsilon})(w_{\varepsilon}^2)_t\\
&\ =- \alpha \frac{d}{dt}\int_\Omega F_{\varepsilon}(u_{\varepsilon}) w_{\varepsilon}^2+ \alpha \int_\Omega F'_{\varepsilon}(u_{\varepsilon})w_{\varepsilon}^2u_{\varepsilon t}\\
 & \ - \beta \frac{d}{dt}\int_\Omega F_{\varepsilon}(v_{\varepsilon}) w_{\varepsilon}^2+ \beta \int_\Omega F'_{\varepsilon}(v_{\varepsilon})w_{\varepsilon}^2v_{\varepsilon t};
\end{split}
\end{equation*}
that is,
\be\label{int-4.50}
\begin{split}
&2\int_\Omega w^2_{\varepsilon t}+\frac{d}{dt}\int_\Omega\left(|\nabla w_{\varepsilon}|^2+\alpha F_{\varepsilon}(u_{\varepsilon}) w_{\varepsilon}^2+\beta F_{\varepsilon}(v_{\varepsilon}) w_{\varepsilon}^2\right)\\
&\ \ = \alpha \int_\Omega F'_{\varepsilon}(u_{\varepsilon})w_{\varepsilon}^2u_{\varepsilon t}+ \beta \int_\Omega F'_{\varepsilon}(v_{\varepsilon})w_{\varepsilon}^2v_{\varepsilon t}.
\end{split}
\ee
Using the first equation in \eqref{reg-2.11} and integrating by parts, we get
\be\label{int-4.51}\begin{split}
&\int_\Omega F'_{\varepsilon}(u_{\varepsilon})w_{\varepsilon}^2u_{\varepsilon t}\\
&\ =-\int_\Omega F''_{\varepsilon}(u_{\varepsilon})w_{\varepsilon}^2|\nabla u_{\varepsilon}|^2-2\int_\Omega F'_{\varepsilon}(u_{\varepsilon}) w_{\varepsilon}\nabla u_{\varepsilon}\cdot\nabla w_{\varepsilon}\\
 &\ +\chi_1\int_\Omega F'_{\varepsilon}(u_{\varepsilon})F''_{\varepsilon}(u_{\varepsilon})u_{\varepsilon} w_{\varepsilon}^2\nabla u_{\varepsilon}\cdot\nabla w_{\varepsilon} +2\chi_1\int_\Omega \left(F'_{\varepsilon}(u_{\varepsilon})\right)^2u_{\varepsilon}w_{\varepsilon}|\nabla w_{\varepsilon}|^2\\
&\ =:H_1+H_2+H_3+H_4.
\end{split}
\ee
 Since  $0\leq -sF''_{\varepsilon}(s)\leq2$ and $w_{\varepsilon}\leq \|w_0\|_{L^\infty}$  due to \eqref{reg-2.10} and   \ref{reg-2.17}, we   estimate
\be\label{int-4.52}
H_1\leq\|w_0\|_{L^\infty}^2\int_\Omega u_{\varepsilon}|F''_{\varepsilon}(u_{\varepsilon})|\cdot\frac{|\nabla u_{\varepsilon}|^2}{u_{\varepsilon}}\leq 2\|w_0\|_{L^\infty}^2\int_\Omega\frac{|\nabla u_{\varepsilon}|^2}{u_{\varepsilon}}.
\ee
Similarly, by Young's inequality, we estimate $H_2$ as follows:
\be\label{int-4.53}\begin{split}
H_2&\leq\int_\Omega F_{\varepsilon}(u_{\varepsilon})|\nabla w_{\varepsilon}|^2+\|w_0\|_{L^\infty}^2\int_\Omega\frac{u_{\varepsilon} \left(F'_{\varepsilon}(u_{\varepsilon})\right)^2}{F_{\varepsilon}(u_{\varepsilon})}\cdot\frac{|\nabla u_{\varepsilon}|^2}{u_{\varepsilon}}\\
&\leq\int_\Omega F_{\varepsilon}(u_{\varepsilon})|\nabla w_{\varepsilon}|^2+\|w_0\|_{L^\infty}^2\int_\Omega\frac{|\nabla u_{\varepsilon}|^2}{u_{\varepsilon}},
\end{split}
\ee
where we used  the following fact due to the definition of $F_{\varepsilon}$ in  \eqref{reg-2.7}:
\be\label{aid1}
0\leq \frac{s\left(F'_{\varepsilon}(s)\right)^2}{F_\varepsilon(s)}=\begin{cases}
 \frac{\varepsilon s}{(1+\varepsilon s)^2\ln(1+\varepsilon s)},    &\text{if  } n=3, \\[0.2cm]
  \frac{1}{(1+\varepsilon s)^3},    &\text{if  } n=4, 5
 \end{cases}\leq 1.
\ee
Analogously,  the term $H_3$ is bounded according to
\be\label{int-4.54}\begin{split}
H_3&\leq\|w_0\|_{L^\infty}^2\int_\Omega F_{\varepsilon}(u_{\varepsilon})|\nabla w_{\varepsilon}|^2\\
&\ \ \ \ +\frac{\chi^2_1\|w_0\|_{L^\infty}^2}{4}\int_\Omega\frac{u^3_{\varepsilon} \left(F'_\varepsilon(u_\varepsilon) F''_{\varepsilon}(u_{\varepsilon})\right)^2}{F_{\varepsilon}(u_{\varepsilon})}\cdot\frac{|\nabla u_{\varepsilon}|^2}{u_{\varepsilon}}\\
&\leq\|w_0\|_{L^\infty}^2\int_\Omega F_{\varepsilon}(u_{\varepsilon})|\nabla w_{\varepsilon}|^2+\chi^2_1\|w_0\|_{L^\infty}^2\int_\Omega\frac{|\nabla u_{\varepsilon}|^2}{u_{\varepsilon}},
\end{split}
\ee
where we used  the following fact due to the definition of $F_{\varepsilon}$ in  \eqref{reg-2.7}:
$$
0\leq \frac{s^3\left(F'_{\varepsilon}(s)F''_{\varepsilon}(s)\right)^2}{F_\varepsilon(s)}=\begin{cases}
 \frac{(\varepsilon s)^3}{(1+\varepsilon s)^6\ln(1+\varepsilon s)},    &\text{if  } n=3, \\[0.2cm]
  \frac{4\varepsilon s}{(1+\varepsilon s)^9},    &\text{if  } n=4, 5
 \end{cases}\leq 4.
$$
Finally, we use \eqref{aid1} to bound   $H_4$ as
\be\label{int-4.55}
\begin{split}
H_4&\leq 2\chi_1\|w_0\|_{L^\infty}\int_\Omega\frac{u_{\varepsilon} \left(F'_{\varepsilon}(u_{\varepsilon})\right)^2 }{F_{\varepsilon}(u_{\varepsilon})}\cdot F_{\varepsilon}(u_{\varepsilon})|\nabla w_{\varepsilon}|^2\\
&\ \leq 2\chi_1\|w_0\|_{L^\infty}\int_\Omega F_{\varepsilon}(u_{\varepsilon})|\nabla w_{\varepsilon}|^2.
\end{split}
\ee
Collecting \eqref{int-4.51},\eqref{int-4.52}, \eqref{int-4.53}, \eqref{int-4.54} and \eqref{int-4.55}, we obtain that
\be\label{int-4.56}
\begin{split}
\alpha\int_\Omega F'_{\varepsilon}(u_{\varepsilon})w_{\varepsilon}^2u_{\varepsilon t}&\leq\left(3+\chi_1^2\right)\alpha\|w_0\|_{L^\infty}^2\int_\Omega\frac{|\nabla u_{\varepsilon}|^2}{u_{\varepsilon}}\\
&\ \ \ +\left(1+2\chi_1\|w_0\|_{L^\infty}+\|w_0\|_{L^\infty}^2\right)\alpha\int_\Omega F_{\varepsilon}(u_{\varepsilon})|\nabla w_{\varepsilon}|^2.
\end{split}
\ee
In a similar manner, one can show that
\be\label{int-4.57}\begin{split}
\beta\int_\Omega F'_{\varepsilon}(v_{\varepsilon})w_{\varepsilon}^2v_{\varepsilon t}&\leq\left(3+\chi_2^2\right)\beta\|w_0\|_{L^\infty}^2\int_\Omega\frac{|\nabla v_{\varepsilon}|^2}{v_{\varepsilon}}\\
&\ \ \ +\left(1+2\chi_2\|w_0\|_{L^\infty}+\|w_0\|_{L^\infty}^2\right)\beta\int_\Omega F_{\varepsilon}(v_{\varepsilon})|\nabla w_{\varepsilon}|^2.
\end{split}
\ee
For any $0\leq s\leq t$,  integrating \eqref{int-4.50} from $s$ to $t$  and combining \eqref{int-4.56} with \eqref{int-4.57}, we end up with
\be\label{int-4.58}\begin{split}
2\int^t_s\int_\Omega w^2_{\varepsilon t}&\leq\int_\Omega\left(|\nabla w_\varepsilon|^2+\alpha F_{\varepsilon}(u_\varepsilon) w_\varepsilon^2+\beta F_{\varepsilon}(v_\varepsilon) w_\varepsilon^2\right)(\cdot, s)\\
& \ \ \ +\left(3+\chi_1^2\right)\alpha\|w_0\|_{L^\infty}^2\int_s^t\int_\Omega\frac{|\nabla u_{\varepsilon}|^2}{u_{\varepsilon}}\\
&\ \ \ +\left(3+\chi_2^2\right)\beta\|w_0\|_{L^\infty}^2\int_s^t\int_\Omega\frac{|\nabla v_{\varepsilon}|^2}{v_{\varepsilon}}\\
& \ \ \ +\left(1+2\chi_1\|w_0\|_{L^\infty}+\|w_0\|_{L^\infty}^2\right)\alpha\int^t_s\int_\Omega F_{\varepsilon}(u_{\varepsilon})|\nabla w_{\varepsilon}|^2\\
& \ \ \ +\left(1+2\chi_2\|w_0\|_{L^\infty}+\|w_0\|_{L^\infty}^2\right)\beta\int^t_s\int_\Omega F_{\varepsilon}(v_{\varepsilon})|\nabla w_{\varepsilon}|^2.
\end{split}
\ee
Using the boundedness of $\|\nabla w_\varepsilon\|_{L^2}$ in \eqref{ulnu+gradv-int-eps} and the conservations of $u_\varepsilon$ and $ v_\varepsilon$,  $0\leq F_\varepsilon (s)\leq s$ and $w_{\varepsilon}\leq \|w_0\|_{L^\infty}$, cf.  \eqref{reg-2.10}, \ref{reg-2.17} and  \eqref{reg-2.15}, we see that
$$
\int_\Omega\left(|\nabla w_\varepsilon|^2+\alpha F_{\varepsilon}(u_\varepsilon) w_\varepsilon^2+\beta F_{\varepsilon}(v_\varepsilon) w_\varepsilon^2\right)(\cdot, s)\leq  K_2+\left(\alpha \|u_0\|_{L^1}+\beta \|v_0\|_{L^1}\right)\|w_0\|_{L^\infty}^2.
$$
Inserting this into \eqref{int-4.58} and using  \eqref{int-4.16} or  \eqref{st-comp-est}, we accomplish  \eqref{int-4.49}  or \eqref{int-4.490}.
\end{proof}

\begin{lemma}\label{lemma-4.6} For $m>\frac{n}{2}+1$, there exists $K_{10}=K_{10}(u_0,v_0,w_0)>0$ such that  the global solution of \eqref{reg-2.11} fulfills, for all $\varepsilon\in(0,1)$,
\be\label{int-4.72}
\int^\infty_0\left(\|u_{\varepsilon t}(\cdot,t)\|^2_{(W^{m,2})^\ast}+\|v_{\varepsilon t}(\cdot,t)\|^2_{(W^{m,2})^\ast}\right)\leq K_{10},  \ \ \   \text{if  } \Omega  \text{ is convex};
\ee
and, there exists  $K_{11}=K_{11}(u_0, v_0, w_0)>0$ such that, for any  $\varepsilon\in(0,1)$ and $t\geq 0$,
\be\label{int-4.73}
\int_t^{t+1}\left(\|u_{\varepsilon t}(\cdot,s)\|^2_{(W^{m,2})^\ast}+\|v_{\varepsilon t}(\cdot,s)\|^2_{(W^{m,2})^\ast}\right)\leq K_{11}, \    \text{if  } \Omega  \text{ is non-convex}.
\ee
\end{lemma}
\begin{proof} For given $\varphi\in W^{m,2}$, we multiply  the first equation in \eqref{reg-2.11} by  $\varphi$, and integrate over $\Omega$ by parts  and use \eqref{aid1} to get
\begin{align*}
\left|\int_{\Omega}u_{\varepsilon t}\varphi\right|&=\left|-\int_{\Omega}\nabla u_{\varepsilon}\cdot\nabla\varphi+\chi_1\int_{\Omega}u_{\varepsilon}F'_{\varepsilon}(
u_{\varepsilon})\nabla w_{\varepsilon}\cdot\nabla\varphi\right|\\
&\leq\Bigr(\int_{\Omega}\frac{|\nabla u_{\varepsilon}|^2}{u_{\varepsilon}}\Bigr)^{\frac{1}{2}}\cdot\Bigr(\int_{\Omega}u_{\varepsilon} |\nabla\varphi|^2\Bigr)^{\frac{1}{2}}\\
&\ \ \ +\chi_1\Bigr(\int_{\Omega}F_{\varepsilon}(u_{\varepsilon})|\nabla w_{\varepsilon}|^2\Bigr)^{\frac{1}{2}}\cdot\Bigr(\int_{\Omega}
\frac{u_{\varepsilon}^2F'_{\varepsilon}(u_{\varepsilon})^2}{F_{\varepsilon}(u_{\varepsilon})}|\nabla\varphi|^2\Bigr)^{\frac{1}{2}}\\
&\leq\left(\int_{\Omega}\frac{|\nabla u_{\varepsilon}|^2}{u_{\varepsilon}}\right)^{\frac{1}{2}} \|u_0\|_{L^1}^\frac{1}{2} \|\nabla\varphi\|_{L^{\infty}}\\
&\ \ +\chi_1\left(\int_{\Omega}F_{\varepsilon}(u_{\varepsilon})|\nabla w_{\varepsilon}|^2\right)^{\frac{1}{2}}\|u_0\|_{L^1}^\frac{1}{2} \|\nabla\varphi\|_{L^{\infty}}.
\end{align*}
Hence,  the  Sobolev embedding $W^{m,2}(\Omega)\hookrightarrow W^{1,\infty}(\Omega)$ due to $m>1+\frac{n}{2}$ shows that
$$
\|u_{\varepsilon t}(\cdot,t)\|^2_{(W^{m,2})^\ast}\leq C_1\int_{\Omega}\frac{|\nabla u_{\varepsilon}|^2}{u_{\varepsilon}}+
C_1\int_{\Omega}F_{\varepsilon}(u_{\varepsilon})|\nabla w_{\varepsilon}|^2, \ \ \ \forall t>0.
$$
 Likewise,
$$
\|v_{\varepsilon t}(\cdot,t)\|^2_{(W^{m,2})^\ast}\leq C_2\int_{\Omega}\frac{|\nabla v_{\varepsilon}|^2}{v_{\varepsilon}}+
C_2\int_{\Omega}F_{\varepsilon}(v_{\varepsilon})|\nabla w_{\varepsilon}|^2, \ \ \ \forall t>0.
$$
By these two inequalities, we readily conclude \eqref{int-4.72} or \eqref{int-4.73}  from  \eqref{int-4.16} or  \eqref{st-comp-est}.
\end{proof}

\begin{lemma}\label{lemma-4.7} For  $n\leq6$,   there    exists  $K_{12}=K_{12}(u_0, v_0, w_0)>0$ such that  the global solution of \eqref{reg-2.11} satisfies, for any  $\varepsilon\in(0,1)$ and $t> 0$,
\be\label{int-4.76}
\int_t^{t+1} \left(\|u_{\varepsilon t}(\cdot,s)\|_{(W^{1,\infty})^\ast}+
\|v_{\varepsilon t}(\cdot,s)\|_{(W^{1,\infty})^\ast}+\|w_{\varepsilon t}(\cdot,s)\|_{(W^{1,\infty})^\ast}\right)\leq K_{12}. \ee
\end{lemma}
\begin{proof}
For   any given $\psi\in W^{1,\infty}$ with  $\|\psi\|_{W^{1,\infty}}\leq1$, notice that $0\leq F'\leq 1$ by \eqref{reg-2.9} and $\frac{n+2}{2}\leq4$ since $n\leq6$, and so by Young's inequality, we get that
\begin{align*}
\Bigr|\int_{\Omega}u_{\varepsilon t}\psi\Bigr|&=\Bigr|-\int_{\Omega}\nabla u_{\varepsilon}\cdot\nabla\varphi+\chi_1\int_{\Omega}u_{\varepsilon}F'_{\varepsilon}(
u_{\varepsilon})\nabla w_{\varepsilon}\cdot\nabla\psi\Bigr|\\
&\leq\int_{\Omega}|\nabla u_{\varepsilon}|+\chi_1\int_{\Omega}u_{\varepsilon}|\nabla w_{\varepsilon}|\\
&\leq\int_{\Omega}|\nabla u_{\varepsilon}|^{\frac{n+2}{n+1}}+|\Omega|+\chi_1\int_{\Omega}u_{\varepsilon}^{\frac{n+2}{n}}+\chi_1\int_{\Omega}
|\nabla w_{\varepsilon}|^{\frac{n+2}{2}}\\
&\leq\int_{\Omega}|\nabla u_{\varepsilon}|^{\frac{n+2}{n+1}}+(1+\chi_1)|\Omega|+\chi_1\int_{\Omega}u_{\varepsilon}^{\frac{n+2}{n}}+\chi_1\int_{\Omega}
|\nabla w_{\varepsilon}|^4,
\end{align*}
 which, upon being integrated from $t$ to $t+1$, gives rise to
 \be\label{u-epsi-func1}
\int_t^{t+1}\|u_{\varepsilon t}\|_{(W^{1,\infty})^\ast} \leq C_1\left(1+\int_t^{t+1}\int_{\Omega} \left(u_{\varepsilon}^{\frac{n+2}{n}}+|\nabla u_{\varepsilon}|^{\frac{n+2}{n+1}}+|\nabla w_{\varepsilon}|^4\right)\right).
\ee
In the same manner, we have
 \be\label{v-epsi-func1}
\int_t^{t+1}\|v_{\varepsilon t}\|_{(W^{1,\infty})^\ast} \leq C_2\left(1+\int_t^{t+1}\int_{\Omega} \left(v_{\varepsilon}^{\frac{n+2}{n}}+|\nabla v_{\varepsilon}|^{\frac{n+2}{n+1}}+|\nabla w_{\varepsilon}|^4\right)\right),
\ee
 and, using the third equation in \eqref{reg-2.11} and the facts that $\|u_\varepsilon\|_{L^1}=\|u_0\|_{L^1}$, $\|v_\varepsilon\|_{L^1}=\|v_0\|_{L^1}$,  $0\leq F_\varepsilon (s)\leq s$ and $w_{\varepsilon}\leq \|w_0\|_{L^\infty}$, cf.  \eqref{reg-2.10}, \ref{reg-2.17} and  \eqref{reg-2.15}, we deduce that
\begin{align*}
\Bigr|\int_{\Omega}w_{\varepsilon t}\psi\Bigr|&=\Bigr|-\int_{\Omega}\nabla w_{\varepsilon}\cdot\nabla\varphi-\alpha\int_{\Omega}F_{\varepsilon}(u_{\varepsilon})w_{\varepsilon}\psi-\beta\int_{\Omega}F_{\varepsilon}
(v_{\varepsilon})w_{\varepsilon}\psi\Bigr|\\
&\leq \int_{\Omega}|\nabla w_{\varepsilon}|+\alpha\int_{\Omega} u_{\varepsilon} w_{\varepsilon}+\beta \int_{\Omega} v_{\varepsilon}  w_{\varepsilon}\\
&\leq \int_{\Omega}|\nabla w_{\varepsilon}|^4+|\Omega|+\left(\alpha\|u_0\|_{L^1}+\beta \|v_0\|_{L^1}\right)\|w_0\|_{L^\infty},
\end{align*}
entailing
 \be\label{w-epsi-func1}
\int_t^{t+1}\|w_{\varepsilon t}\|_{(W^{1,\infty})^\ast}\leq C_3\left(1+\int_t^{t+1}\int_{\Omega}|\nabla w_{\varepsilon}|^4 \right).
\ee
The desired estimate \eqref{int-4.76} follows from a combination of \eqref{u-epsi-func1}, \eqref{v-epsi-func1} and \eqref{w-epsi-func1} with Lemmas \ref{lemma-4.2} and \ref{lemma-4.3}.
\end{proof}
\subsection{Global existence of weak solutions in 3,4,5 D}
With the estimates gained in last subsection, we can  first derive strong compactness properties by means of the Aubin-Lions and then  obtain the existence of weak solutions in  in 3, 4 and 5D via extraction procedure, cf. \cite{TW12-JDE, Win12-CPDE, Win17-JDE}.
\begin{lemma}\label{lemma-4.8} For  $n\in\{3,4,5\}$, there exist $(\varepsilon_j)_{j\in\mathbb{N}}\subset(0,1)$ and nonnegative functions $u,v$ and $w$ satisfying \eqref{data-reg} such that $\varepsilon_j\searrow0$ as $j\rightarrow\infty$ and that the global solution of \eqref{reg-2.11} satisfies, as $\varepsilon=\varepsilon_j\searrow0$, that
\be\label{int-4.79}
\left(u_{\varepsilon}, \  v_{\varepsilon}, \  w_{\varepsilon}\right) \rightarrow \left(u, \  v, \ w\right)\ \ \ a.e. \ in \ \Omega\times(0,\infty),
\ee
\be\label{int-4.80}
\left(u_{\varepsilon}, \  F_{\varepsilon}(u_{\varepsilon}), \  u_{\varepsilon}F'_{\varepsilon}(u_{\varepsilon})\right) \rightarrow \left(u, \    u,  \ u\right)  \mbox{in}  \ L^p_{\mbox{loc}}(\overline{\Omega}\times[0,\infty)), \ \forall  p\in[1,\frac{n+2}{n}),
\ee
\be\label{int-4.83}
\left(v_{\varepsilon}, \  F_{\varepsilon}(v_{\varepsilon}), \  v_{\varepsilon}F'_{\varepsilon}(v_{\varepsilon})\right) \rightarrow \left(v, \    v,  \  v\right)  \mbox{in}  \ L^p_{\mbox{loc}}(\overline{\Omega}\times[0,\infty)),\  \forall  p\in[1, \frac{n+2}{n}),
\ee
\be\label{int-4.81}
\left(\nabla u_{\varepsilon}, \  \nabla v_{\varepsilon}\right) \rightharpoonup\left(\nabla u, \ \nabla v\right) \    \mbox{ in }   L^{\frac{n+2}{n+1}}_{\mbox{loc}}(\overline{\Omega}\times[0,\infty)),
\ee
\be\label{int-4.86}
 w_{\varepsilon}(\cdot,t)\rightarrow w(\cdot,t) \ \ in \ L^q(\Omega) \ for \  a.e.  \ t\in(0,\infty),
\ee
\be\label{int-4.87}
w_{\varepsilon}\stackrel{*}{\rightharpoonup} w \ \ in \ L^{\infty}_{loc}(\overline{\Omega}\times[0,\infty)),  \   \nabla w_{\varepsilon}\rightharpoonup\nabla w \ \ in \ L^4_{loc}(\overline{\Omega}\times[0,\infty)) \ and
\ee
\be\label{int-4.88}
 \left(u_{\varepsilon}F'_{\varepsilon}(u_{\varepsilon})\nabla w_{\varepsilon},\
  v_{\varepsilon}F'_{\varepsilon}(v_{\varepsilon})\nabla w_{\varepsilon}\right)\rightharpoonup \left(u\nabla w, \ v\nabla w\right)~~~in~L^1_{loc}(\overline{\Omega}\times[0,\infty)).
\ee
Here, the exponent $q$ is related to the Sobolev conjugate number $\frac{4n}{(n-4)^+}$ and satisfies
\be\label{q-exponent}
q\in [1,\infty] \text{  if  } n=3;  \  \  q\in [1,\infty) \text{  if    } n=4; \ \  q\in [1,20) \text{  if   } n=5.
\ee
Moreover, $(u,v,w)$ is a global weak solution of \eqref{PPT} in the sense of Definition \ref{weak-uvw}.
\end{lemma}
\begin{proof}
By Lemmas \ref{lemma-2.4}, \ref{lemma-2.5},  \ref{lemma-4.1}, \ref{lemma-4.2}, \ref{lemma-4.3} and \ref{lemma-4.7}, for each $T>0$, we know that
$$
(u_{\varepsilon})_{\varepsilon\in(0,1)}~~and~~(v_{\varepsilon})_{\varepsilon\in(0,1)}~~~are~bounded~in~L^{\frac{n+2}{n+1}}((0,T);W^{1,\frac{n+2}{n+1}}(\Omega))
$$
and
$$
(u_{\varepsilon t})_{\varepsilon\in(0,1)}~~and~~(v_{\varepsilon t})_{\varepsilon\in(0,1)}~~~are~bounded~in~L^1((0,T);(W^{1,\infty}(\Omega))^\ast)
$$
as well as
$$
(w_{\varepsilon})_{\varepsilon\in(0,1)}~~~is~bounded~in~L^4((0,T);W^{1,4}(\Omega))
$$
and
$$
(w_{\varepsilon t})_{\varepsilon\in(0,1)}~~~is~bounded~in~L^1((0,T);(W^{1,\infty}(\Omega))^\ast).
$$
By the compact embeddings $W^{1,\frac{n+2}{n+1}}(\Omega)\hookrightarrow L^1(\Omega)$ and $W^{1,4}(\Omega)\hookrightarrow L^q(\Omega)$ for all $q\in[1,\infty]$ with  $1-\frac{n}{4}> -\frac{n}{q}$, twice direct applications of the  Aubin-Lions lemma  \cite{Sim97}, we see there exist $(\varepsilon_j)_{j\in\mathbb{N}}\subset(0,1)$ and nonnegative functions $u,v$ and $w$ fulfilling \eqref{data-reg} and such that,  for any such $q$ and   $\varepsilon=\varepsilon_j\searrow0$,
\be\label{int-4.89}
u_{\varepsilon}\rightarrow u \ \ \  and \ \ \ v_{\varepsilon}\rightarrow v \ \  in \ L^1(\Omega\times(0,T))
\ee
as well as
\be\label{int-4.90}
w_{\varepsilon}\rightarrow w \ \   in \ L^4((0,T);L^q(\Omega)).
\ee
 Now, based on \eqref{int-4.89} and \eqref{int-4.90}, upon  passing to a subsequence if necessary,  we can infer that \eqref{int-4.79} and \eqref{int-4.86} hold, and also that \eqref{int-4.81} and  \eqref{int-4.87}  hold. Using the a.e. convergence in \eqref{int-4.79}, the boundedness in \eqref{int-4.29} and the properties of $F_\varepsilon$ in \eqref{reg-2.9} and \eqref{reg-2.10}, we use  the Vitali convergence theorem (roughly, a.e. convergence plus uniform integrability imply convergence) to infer   \eqref{int-4.80} and \eqref{int-4.83}. Using the convergence features in  \eqref{int-4.80},  \eqref{int-4.83} and \eqref{int-4.87} and noting the fact that $\frac{n}{n+2}+\frac{1}{4}<1$ since $n\leq5$, we conclude  \eqref{int-4.88}, which together with \eqref{data-reg} implies  that the regularity conditions in Definition \ref{weak-uvw} are fulfilled. The reminding verifications of  the integration by parts identities   \eqref{reg-2.4}, \eqref{reg-2.5} and \eqref{reg-2.6} can be easily adapted from   \cite{Ren-Liu-1, Ren-Liu-4,  Win17-JDE}.
\end{proof}
\begin{proof}[\bf{Proof of 3,4, 5D global existence of weak solutions}]
The statement on  3,4 and 5D global existence of weak solutions  has been fully contained in  Lemma \ref{lemma-4.8}.
\end{proof}
\subsection{Large time behavior  of global weak solutions in convex domains}
In this subsection, we focus on the eventual smoothness and stabilization of global weak solutions to \eqref{PPT}, that is, the limiting functions $(u,v,w)$ of $(u_\varepsilon, v_\varepsilon, w_\varepsilon)$.
\begin{lemma}\label{lemma-4.9} For $n\in\{3,4,5\}$, there exists $(\varepsilon_j)_{j\in\mathbb{N}}\subset(0,1)$ of numbers $\varepsilon_j\searrow0$ such that as $\varepsilon=\varepsilon_j\searrow0$, the global solution of \eqref{reg-2.11} fulfill the following properties:
\be\label{int-4.59}
w\in C^0([0,\infty);L^2(\Omega))
\ee
and
\be\label{int-4.62}
w_{\varepsilon}\rightarrow w \ \ in \ L^{\infty}_{\mbox{loc}}([0,\infty);L^2(\Omega)).
\ee
\end{lemma}
\begin{proof}
In light of the $\varepsilon$-independent estimates  provided in Lemmas  \ref{lemma-2.5} and  \ref{lemma-4.5}, we shall adapt the arguments in  \cite[Corollary 5.3]{TW12-JDE} to derive \eqref{int-4.59} and  \eqref{int-4.62}. As a matter of fact, for  any $T>0$, it follows from Lemmas \ref{lemma-2.5} and \ref{lemma-4.5} that  $(w_\varepsilon)_{\varepsilon\in (0, 1)}$ is bounded in  $L^\infty((0, T); L^\infty(\Omega))$  and that  $(w_{\varepsilon t})_{\varepsilon\in (0, 1)}$ is bounded in $L^2((0, T); L^2(\Omega))$. Then, arguing as  \cite[Corollary 5.3]{TW12-JDE}, we see that
$$
\sup_{t\in (0, T)}\|w_{\varepsilon}(\cdot, t)\|_{L^2}+\sup_{t\neq s,\
t, s\in (0, T)}\frac{\left|\|w_{\varepsilon}(\cdot, t)\|_{L^2}-\|w_{\varepsilon}(\cdot, s)\|_{L^2}\right|}{|t-s|^\frac{1}{2}}\leq C_1,
$$
 that is, $(w_\varepsilon)_{\varepsilon\in (0, 1)}$ is bounded in   $C^\frac{1}{2}([0, T]; L^2(\Omega))$, and thus is relatively compact in   $C([0, T]; L^2(\Omega))$ by the Arzela and Ascoli compactness theorem. This establishes \eqref{int-4.62} and  thus \eqref{int-4.59}.
\end{proof}
In the sequel, we fix $(\varepsilon=\varepsilon_j)_{j\in\mathbb{N}}$ so that Lemmas \ref{lemma-4.8} and \ref{lemma-4.9} hold.  With the strong convergence provided in Lemma \ref{lemma-4.14} below, the following lemma follows.

\begin{lemma}\label{lemma-4.10} With $(u_\varepsilon, v_\varepsilon, w_\varepsilon)$ replaced by  $(u,v,w)$ constructed in Lemma \ref{lemma-4.8},  the conclusions of Lemmas \ref{lemma-2.4}, \ref{lemma-2.5} and \ref{lemma-4.2} through \ref{lemma-4.7} still hold.
\end{lemma}
\begin{lemma}\label{lemma-4.11} The    global weak solution of \eqref{PPT} constructed in Lemma \ref{lemma-4.8} fulfills
\be\label{int-4.63}
 w(\cdot,t)\rightarrow 0 \ \ in \ L^\infty(\Omega) \ as \ t\rightarrow\infty.
\ee
\end{lemma}
\begin{proof}
We   shall  extend the arguments in   \cite[Lemma 5.4]{TW12-JDE} (with a small flaw)  for $n=3$ to higher dimensional cases. Notice from Lemmas \ref{lemma-2.5} and \ref{lemma-4.2} that
$$
\|w\|_{L^\infty(\Omega\times(0,\infty))}+\int_t^{t+1}\int_\Omega |\nabla w|^4\leq C_1,  \quad \quad \forall t\geq 0.
$$
This ensures the existence of a sequence times $t_k\rightarrow\infty$ with $1+t_k\leq t_{k+1}\leq 2+t_k$ such that $(w(\cdot, t_k))_{k\in\mathbb{N}}$ is bounded in $W^{1,4}(\Omega)$. This together with  the compact embedding from $W^{1,4}(\Omega)$ into $L^q(\Omega)$ for $q$ satisfying \eqref{q-exponent}   allows  us to deduce, up to a subsequence, for some nonnegative function $w_\infty$, that
\be\label{w-lim1}
w(\cdot, t_k)\rightarrow w_\infty \ \ in \ L^q(\Omega) \ as \ k\rightarrow\infty.
\ee
Using  Cauchy-Schwarz inequality, we infer that
\begin{align*}
&\int_{t_k}^{1+t_k}\int_\Omega \left|w_\varepsilon(\cdot, t)-w_\varepsilon(\cdot, t_k)\right|^2\\
&=\int_{t_k}^{1+t_k}\int_\Omega \left(\int_{t_k}^tw_{\varepsilon t}(\cdot, s)\right)^2\leq  \int_{t_k}^{1+t_k}\int_\Omega  w_{\varepsilon t}^2(\cdot, t), \quad  \forall \varepsilon\in (0, 1).
\end{align*}
In view of  Lemma \ref{lemma-4.5}, upon passing to the limit $\varepsilon=\varepsilon_j\searrow0$, this gives rise to
$$
\int_{t_k}^{1+t_k}\int_\Omega \left|w(\cdot, t)-w(\cdot, t_k)\right|^2
\leq  \int_{t_k}^{1+t_k}\int_\Omega  w_{t}^2(\cdot, t)\rightarrow 0 \text{  as }  k\rightarrow \infty.
$$
This together with \eqref{w-lim1} with $q=2$  implies that
\be\label{w-lim2}
\begin{split}
&\frac{1}{2}\int_{t_k}^{1+t_k}\int_\Omega \left|w(\cdot, t)-w_\infty\right|^2\\
&\leq \|w(\cdot, t_k)-w_\infty\|_{L^2}^2+\int_{t_k}^{1+t_k}\int_\Omega \left|w(\cdot, t)-w(\cdot, t_k)\right|^2\rightarrow 0 \text{  as }  k\rightarrow \infty.
\end{split}
\ee
Extracting  Lemma \ref{lemma-4.4} from Lemma \ref{lemma-4.10}, we have
\be\label{w-lim3}
\int_{t_k}^{1+t_k}\left(\|u-\bar{u}_0\|^2_{L^\frac{n}{n-1}}
+\|v-\bar{v}_0\|^2_{L^\frac{n}{n-1}}\right)\rightarrow 0 \text{  as }  k\rightarrow \infty.
\ee
Since  $w_{\varepsilon}\leq \|w_0\|_{L^\infty}$, we use  H\"{o}lder's inequality to estimate, for all $k\in\mathbb{N}$, that
\begin{align*}
&\int_{t_k}^{1+t_k}\int_\Omega \left|\left(\alpha u+\beta v\right)w-\left(\alpha \bar{u}_0+\beta \bar{v}_0\right)w_\infty\right|\\
&\leq \int_{t_k}^{1+t_k}\int_\Omega \left(\alpha |u-\bar{u}_0|+\beta|v-\bar{v}_0|\right)w
+\left(\alpha \bar{u}_0+\beta \bar{v_0}\right)\int_{t_k}^{1+t_k}\int_\Omega|w-w_\infty| \\
&\leq  \left[\alpha \left(\int_{t_k}^{1+t_k}\|u-\bar{u}_0\|^2_{L^\frac{n}{n-1}}\right)^\frac{1}{2}
+ \beta \left(\int_{t_k}^{1+t_k}\|v-\bar{v}_0\|^2_{L^\frac{n}{n-1}}\right)^\frac{1}{2}\right]
\left(\int_{t_k}^{1+t_k}\|w\|^2_{L^n}\right)^\frac{1}{2}\\
&\ \ + \left(\alpha \bar{u}_0+\beta \bar{v_0}\right)|\Omega|^\frac{1}{2}
\left(\int_{t_k}^{1+t_k}\|w-w_\infty\|^2_{L^2}\right)^\frac{1}{2}\\
&\ \leq  \left[\alpha \left(\int_{t_k}^{1+t_k}\|u-\bar{u}_0\|^2_{L^\frac{n}{n-1}}\right)^\frac{1}{2}
+ \beta \left(\int_{t_k}^{1+t_k}\|v-\bar{v}_0\|^2_{L^\frac{n}{n-1}}\right)^\frac{1}{2}\right]
\|w_0\|_{L^\infty}|\Omega|^\frac{1}{2}\\
&\ \ + \left(\alpha \bar{u}_0+\beta \bar{v_0}\right)|\Omega|^\frac{1}{2}
\left(\int_{t_k}^{1+t_k}\|w-w_\infty\|^2_{L^2}\right)^\frac{1}{2},
\end{align*}
and so, we obtain from \eqref{w-lim2} and \eqref{w-lim3} that
\be\label{w-lim4}
\int_{t_k}^{1+t_k}\int_\Omega  \left(\alpha u+\beta v\right)w\rightarrow  \left(\alpha \bar{u}_0+\beta \bar{v}_0\right)\int_\Omega  w_\infty\  \  \text{ as }  k\rightarrow\infty.
\ee
On the other hand, by \eqref{reg-2.19} and the fact $1+t_k\leq t_{k+1}$, we have
$$
\sum_{k=1}^\infty\int_{t_k}^{1+t_k}\int_\Omega  \left(\alpha u+\beta v\right)w\leq \int_0^\infty\int_\Omega  \left(\alpha u+\beta v\right)w<+\infty.
$$
Since  $\bar{u}_0, \bar{v}_0>0$,  this couples with \eqref{w-lim4} implies  $w_\infty\equiv 0$, and so \eqref{w-lim1} becomes
$$
w(\cdot, t_k)\rightarrow 0 \ \ in \ L^q(\Omega) \ as \ k\rightarrow\infty.
$$
This together with the fact that   $t\mapsto \|w(\cdot, t)\|_{L^q}$ is non-increasing by  Lemma   \ref{lemma-2.5} shows actually that
\be\label{w-lim5}
w(\cdot, t)\rightarrow 0 \ \ in \ L^q(\Omega) \ as \ k\rightarrow\infty.
\ee
Since $\|w(\cdot, t)\|_{L^\infty}$ is non-increasing in $t$ and is nonnegative, we see, as $t\rightarrow \infty$,  that $\|w(\cdot, t)\|_{L^\infty}$ converges decreasingly to some  $a\geq 0$. If $a>0$, then, for any $\eta\in (0, a)$ and $t>0$,  we define $\Omega(t)=\{x\in \Omega: w(x,t)\geq a-\eta\}$, and thus we get
$$
\|w(\cdot, t)\|_{L^q}=\left(\int_\Omega w^q(\cdot,t)\right)^\frac{1}{q}\geq (a-\eta)|\Omega(t)|^\frac{1}{q},
$$
which couples with \eqref{w-lim5} immediately yields $\lim_{t\rightarrow \infty}|\Omega(t)|=0$. While, this is incompatible with the fact that  $ \|w(\cdot, t)\|_{L^\infty}\geq a$ for all $t\geq0$. Therefore, we must have $a=0$, and so \eqref{int-4.63} follows. \end{proof}

\begin{lemma}\label{lemma-4.12} For $n\in\{3,4,5\}$ and for any $\delta>0$, there exist $t_0(\delta)>0$ and $\varepsilon_0(\delta)\in(0,1)$ such that for all $\varepsilon\in(\varepsilon_j)_{j\in\mathbb{N}}$ satisfying $\varepsilon<\varepsilon_0(\delta)$, the $w$-component of the global  classical solution of \eqref{reg-2.11} fulfills
\be\label{int-4.64}
0\leq  w_{\varepsilon}  \leq \delta  \text{ on  } \Omega\times ( t_0(\delta), \infty).
\ee
\end{lemma}
\begin{proof}
For given $\delta>0$ and $q$ satisfying \eqref{q-exponent}, it follows from Lemma \ref{lemma-4.11} there exists $\hat{t}_0>0$ such that the limit $w$ defined by Lemma \eqref{lemma-4.8} fulfills $\|w(\cdot, t)\|_{L^q}\leq\delta/2$ for $t>\hat{t}_0$. Thanks to  \eqref{int-4.86},  there exists $t_0\in(\hat{t}_0,\hat{t}_0+1)$ such that
$w_{\varepsilon}(\cdot,t_0)\rightarrow w(\cdot,t_0)$ in $L^q(\Omega)$ as $\varepsilon=\varepsilon_j\searrow0$. This joined  with the fact from Lemma \ref{lemma-2.5}  that $t\mapsto \|w(\cdot, t)\|_{L^q}$ is non-increasing ensures   there exists $\varepsilon_0(\delta)>0$ such that, for $t\geq t_0$ and $\varepsilon\in(\varepsilon_j)_{j\in\mathbb{N}}$ with $\varepsilon<\varepsilon_0(\delta)$,
\be\label{w-lim6}
\begin{split}
\|w_{\varepsilon}(\cdot,t)\|_{L^q}\leq \|w_{\varepsilon}(\cdot,t_0)\|_{L^q}&\leq\|w_{\varepsilon}(\cdot,t_0)-w(\cdot,t_0)\|_{L^q}
+\|w(\cdot,t_0)\|_{L^q}\\
&\ \leq \frac{\delta}{2}+ \|w(\cdot,t_0)\|_{L^q}\leq\delta.
\end{split}
\ee
Now, let $a=\limsup_{(\varepsilon, t)\rightarrow (0, \infty)}\|w_\varepsilon(\cdot,t)\|_{L^\infty}$; if $a>0$, then, for any $\eta\in(0, a)$, we define  $E_t^\varepsilon=\{x\in \Omega: w_\varepsilon(x,t)\geq a-\eta\}$, and then we derive
$$
\|w_\varepsilon(\cdot, t)\|_{L^q}=\left(\int_\Omega w_\varepsilon^q(\cdot,t)\right)^\frac{1}{q}\geq (a-\eta)|E_t^\varepsilon|^\frac{1}{q},
$$
which couples with \eqref{w-lim6} immediately yields $\limsup_{(\varepsilon, t)\rightarrow (0, \infty)}|E_t^\varepsilon|=0$. While, this is contradictory to the definitions of $a$ and $E_t^\varepsilon$, and therefore, we must have $a=0$, and so the  decay estimate \eqref{int-4.64} follows.
\end{proof}

\begin{lemma}\label{lemma-4.13} For $n\in\{3,4,5\}$ and   $p\in(1,\infty)$, there  exist  $t_1(p)>0$, $\varepsilon_1(p)\in(0,1)$ and $K_{13}(p)=K_{13}(u_0, v_0, w_0, p)>0$  such that for $\varepsilon\in(\varepsilon_j)_{j\in\mathbb{N}}$ with $\varepsilon<\varepsilon_1(p)$, the global  classical solution of \eqref{reg-2.11}  satisfies
\be\label{int-4.65}
\int_{\Omega}u^p_{\varepsilon}(\cdot,t)+\int_{\Omega}v^p_{\varepsilon}(\cdot,t)\leq K_{13}(p),  \ \ \ \forall t\geq t_1(p).
\ee
\end{lemma}
\begin{proof}
With Lemmas \ref{lemma-4.11} and \ref{lemma-4.12} at hand, we can extend the arguments in   \cite[Lemma 6.2]{TW12-JDE} for $n=3$ to its higher dimensional cases \eqref{int-4.65}. To this end, for $p\in(1, \infty)$ and for  $r\in (0, p-1)$, we first see that
\begin{align*}
&\left\{0\leq z<\frac{r+1}{p}, \ \ \   p-1-\frac{p}{4r}\cdot\frac{4r^2+(p-1)^2z^2}{r+1
-  pz}>0\right\}\\
&\Longleftrightarrow 0\leq z<\frac{2(p-1-r)\sqrt{pr}}{(p-1)p\left(\sqrt{(p-1)(r+1)}+\sqrt{pr}\right)}
\end{align*}
and
$$
\frac{2(p-1-r)\sqrt{pr}}{(p-1)p\left(\sqrt{(p-1)(r+1)}+\sqrt{pr}\right)}
>\frac{2(p-1-r)r}{p^2\left(p+p\right)}=\frac{(p-1-r)r}{p^3},
$$ and thus, we see, for
\be\label{delta-def}
\begin{split}
\delta&=\min\left\{1, \ \frac{1}{\chi_1}, \  \frac{1}{\chi_2}\right\} \min\left\{ \frac{(p-1-r)r}{2p^3}, \ \frac{r+1}{2p}\right\}\\
& \ =\min\left\{1, \ \frac{1}{\chi_1}, \  \frac{1}{\chi_2}\right\} \frac{(p-1-r)r}{2p^3}>0,
\end{split}
\ee
we have that
\be\label{delta-equ2}
A_i:=p(2\delta)^{-r}\left\{p-1-\frac{p}{4r}\cdot\frac{4r^2+(p-1)^2(2\delta \chi_i)^2}{r+1
-2\delta p\chi_i}\right\}>0, \ \  \ i=1, \ 2.
\ee
With these preparations,  we now specify $\phi$ as follows:
$$
\phi(s)=(2\delta-s)^{-r}, \quad \quad s\in[0,  2\delta).
$$
Now, for    $\varepsilon\in(\varepsilon_j)_{j\in\mathbb{N}}$ satisfying $\varepsilon<\varepsilon_0(\delta)$ and $t>t_0(\delta)$  with  $\varepsilon_0(\delta)$ and $t_0(\delta)$ prescribed by Lemma \ref{lemma-4.12}, we see that $\phi(w_\varepsilon)$ is well-defined,  and, moreover,  we use    integration by parts to compute from \eqref{reg-2.11}  that
\be\label{w-lim7}
\begin{split}
&\frac{d}{dt}\int_\Omega \left(u_\varepsilon^p+v_\varepsilon^p\right)\phi(w_\varepsilon)\\
&\ =-(p-1)p\int_\Omega \left(u_\varepsilon^{p-2}|\nabla u_\varepsilon|^2+v_\varepsilon^{p-2}|\nabla v_\varepsilon|^2\right)\phi(w_\varepsilon)\\
&\ \   -\int_\Omega \left[\phi''(w_\varepsilon)-p\chi_1F_\varepsilon'(u_\varepsilon) \phi'(w_\varepsilon) \right]u_\varepsilon^p|\nabla w_\varepsilon|^2\\
&\ \  -\int_\Omega \left[\phi''(w_\varepsilon)-p\chi_2 F_\varepsilon'(v_\varepsilon) \phi'(w_\varepsilon) \right]v_\varepsilon^p|\nabla w_\varepsilon|^2\\
&\ \  +p\int_\Omega u_\varepsilon^{p-1}\left[-2\phi'(w_\varepsilon)+(p-1)\chi_1 F_\varepsilon'(u_\varepsilon) \phi(w_\varepsilon) \right] \nabla u_\varepsilon \nabla w_\varepsilon\\
& \ \   +p\int_\Omega v_\varepsilon^{p-1}\left[-2\phi'(w_\varepsilon)+(p-1)\chi_2 F_\varepsilon'(v_\varepsilon) \phi(w_\varepsilon) \right] \nabla v_\varepsilon \nabla w_\varepsilon\\
&\ \  -\int_\Omega \left[\alpha F_\varepsilon(u_\varepsilon)+\beta F_\varepsilon(v_\varepsilon)\right]\left(u_\varepsilon^p +v_\varepsilon^p\right)\phi'(w_\varepsilon)\\
\ \  &=:I_1+I_2+I_3+I_4+I_5+I_6.
\end{split}
\ee
By the nonnegativity  of $u_\varepsilon,v_\varepsilon,w_\varepsilon, F_\varepsilon, \phi'$,  we first see obviously  that $I_6\leq 0$, and then,  by the facts $0\leq w_{\varepsilon}\leq \delta$ in \eqref{int-4.64}, $ 0\leq F_\varepsilon'(s)\leq 1$ in  \eqref{reg-2.9} and the choice  of $\delta$ in \eqref{delta-def}, we infer that
\begin{align*}
\phi''(w_\varepsilon)-p\chi_1F_\varepsilon'(u_\varepsilon) \phi'(w_\varepsilon)\geq \phi''(w_\varepsilon)-p\chi_1  \phi'(w_\varepsilon) \geq r(2\delta)^{-r-2}\left(r+1-2p\chi_1\delta\right)>0,\\
\phi''(w_\varepsilon)-p\chi_2F_\varepsilon'(v_\varepsilon) \phi'(w_\varepsilon)\geq \phi''(w_\varepsilon)-p\chi_2  \phi'(w_\varepsilon) \geq r(2\delta)^{-r-2}\left(r+1-2p\chi_2\delta\right)>0.
\end{align*}
Hence, we employ Young's inequality to estimate $I_4$ as
\be\label{I4-est}
I_4\leq -I_2+\frac{p^2}{4}\int_\Omega \frac{\left[-2\phi'(w_\varepsilon)+(p-1)\chi_1 F_\varepsilon'(u_\varepsilon) \phi(w_\varepsilon) \right]^2}{\phi''(w_\varepsilon)-p\chi_1 \phi'(w_\varepsilon)} u_\varepsilon^{p-2} |\nabla u_\varepsilon|^2,
\ee
and then, based on \eqref{w-lim7}, we further use the facts $ 0\leq F_\varepsilon'(s)\leq 1$,  the choices  of $\delta$  and $A_1$ in \eqref{delta-def} and \eqref{delta-equ2} to estimate, for $(x,t,s)\in \Omega\times (t_0(\delta), \infty)\times [0,\delta]$,
\begin{align*}
B_1(x,t,s)&:= p(p-1)\phi(s)-\frac{p^2}{4}\cdot \frac{\left[-2\phi'(s)+(p-1)\chi_1 F_\varepsilon'(u_\varepsilon) \phi(s) \right]^2}{\phi''(s)-p\chi_1 \phi'(s)}\\
&\geq p(p-1)\phi(s)-\frac{p^2}{4}\cdot \frac{4\phi'^2(s)+(p-1)^2\chi_1^2 \phi^2(s)}{\phi''(s)-p\chi_1 \phi'(s)}\\
&=p(2\delta-s)^{-r}\left\{p-1-\frac{p}{4r}\cdot\frac{4r^2+(p-1)^2(2\delta-s)^2\chi_1^2}{r+1
-(2\delta-s)p\chi_1}\right\}\\
&\geq p(2\delta)^{-r}\left\{p-1-\frac{p}{4r}\cdot\frac{4r^2+(p-1)^2(2\delta \chi_1)^2}{r+1
-2\delta p\chi_1}\right\}=A_1>0.
\end{align*}
Combining this with \eqref{I4-est}, we conclude that
\be\label{I4-est2}
I_4\leq -I_2+(p-1)p\int_\Omega  \phi(w_\varepsilon) u_\varepsilon^{p-2} |\nabla u_\varepsilon|^2-A_1\int_\Omega  u_\varepsilon^{p-2} |\nabla u_\varepsilon|^2.
\ee
In the same reasoning, we readily estimate the term $I_5$ as
\be\label{I5-est}
I_5\leq -I_3+(p-1)p\int_\Omega  \phi(w_\varepsilon) v_\varepsilon^{p-2} |\nabla v_\varepsilon|^2-A_2\int_\Omega  v_\varepsilon^{p-2} |\nabla v_\varepsilon|^2.
\ee
Substituting \eqref{I4-est}, \eqref{I4-est2} and \eqref{I5-est} into \eqref{w-lim7}  and recalling $I_6\leq0$, we infer that
\be\label{w-lim8}
\begin{split}
 \frac{d}{dt}\int_\Omega \left(u_\varepsilon^p+v_\varepsilon^p\right)\phi(w_\varepsilon)&\leq -A_1\int_\Omega  u_\varepsilon^{p-2}|\nabla u_\varepsilon|^2-A_2\int_\Omega v_\varepsilon^{p-2}|\nabla v_\varepsilon|^2\\
 &=-\frac{4A_1}{p^2}\int_\Omega  |\nabla u_\varepsilon^\frac{p}{2}|^2-\frac{4A_2}{p^2}\int_\Omega  |\nabla v_\varepsilon^\frac{p}{2}|^2.
 \end{split}
\ee
Next, due to the  mass conservations of $u_\varepsilon$ and $v_\varepsilon$  in \eqref{reg-2.15}, by the GN inequality (cf. Lemma \ref{GN-inter}) and the fact that $\phi(w_\varepsilon)\leq \delta^{-r}$, we infer that
\begin{align*}
\int_\Omega  u_\varepsilon^p \phi(w_\varepsilon)\leq \delta^{-r}\|u_\varepsilon^\frac{p}{2}\|_{L^2}^2&\leq C_1\|\nabla u_\varepsilon^\frac{p}{2}\|_{L^2}^\frac{2n(p-1)}{(p-1)n+2}
\|u_\varepsilon^\frac{p}{2}\|_{L^\frac{2}{p}}^\frac{4}{(p-1)n+2}
+C_1\|u_\varepsilon^\frac{p}{2}\|_{L^\frac{2}{p}}^2\\
&\leq C_2\|\nabla u_\varepsilon^\frac{p}{2}\|_{L^2}^\frac{2n(p-1)}{(p-1)n+2}+C_2
\end{align*}
and, similarly, that
 $$\int_\Omega  v_\varepsilon^p \phi(w_\varepsilon)\leq C_3\|\nabla v_\varepsilon^\frac{p}{2}\|_{L^2}^\frac{2n(p-1)}{(p-1)n+2}+C_3.
 $$
 Setting $y_\varepsilon(t)=\int_\Omega \left(u_\varepsilon^p+v_\varepsilon^p\right)\phi(w_\varepsilon)(\cdot,t)$, we derive from \eqref{w-lim8} an ODI as follows:
 $$
 y_\varepsilon'(t)\leq -C_4\left(y_\varepsilon(t)-1\right)_+^\frac{(p-1)n+2}{n(p-1)}\Longleftrightarrow \left[\left(y_\varepsilon(t)-1\right)_+^\frac{-2}{(p-1)n}\right]'\geq \frac{2C_4}{(p-1)n}, \   \  t>t_0(\delta).
 $$
 An integration enables us to deduce that
 $$
 y_\varepsilon(t)\leq 1+\left(\frac{2C_4}{(p-1)n}(t-t_0(\delta))\right)^{-\frac{(p-1)n}{2}}, \ \ \ t>t_0(\delta).
 $$
 Recalling the definition of $y_\varepsilon$ and $\phi(w_\varepsilon)\geq (2\delta)^{-r}$, we immediately arrive at
 $$
 \int_\Omega \left(u_\varepsilon^p+v_\varepsilon^p\right)\leq (2\delta)^ry_\varepsilon(t)\leq (2\delta)^r\left[1+\left(\frac{2C_4}{(p-1)n}\right)^{-\frac{(p-1)n}{2}}\right], \  t\geq 1+t_0(\delta),
 $$
 yielding our desired estimate \eqref{int-4.65} upon setting $t_1(\delta)=1+t_0(\delta)$.
\end{proof}

With the uniform eventual $L^p$-boundedness of $u_\varepsilon$ and $v_\varepsilon$ in Lemma \ref{lemma-4.13}, it is quite standard via bootstrap argument or semigroup technique (cf. \cite{Hor-Win, TW12-JDE} to obtain the uniform eventual $C^2$-boundedness of weak solutions.

\begin{lemma}\label{lemma-4.14} For $n\in\{3,4,5\}$, there    exist  $T>0$ and $K_{14}=K_{14}(u_0, v_0, w_0)>0$  such that for $\varepsilon\in(\varepsilon_{j_i})_{i\in\mathbb{N}}$ of $(\varepsilon_j)_{j\in\mathbb{N}}$, the global  classical solution of \eqref{reg-2.11}  fulfills
\be\label{int-4.66}
\|u_{\varepsilon}(\cdot,t)\|_{C^2(\overline{\Omega})}+\|v_{\varepsilon}(\cdot,t)\|_{C^2(\overline{\Omega})}\leq K_{14},  \ \ \ \forall t\geq T,
\ee
and such that
\be\label{int-4.67}
u_{\varepsilon}\rightarrow u, \ \ v_{\varepsilon}\rightarrow v \ \ and \ \ w_{\varepsilon}\rightarrow w \ in \ C^{2,1}_{loc}(\overline{\Omega}\times[T,\infty)) \ as \ \varepsilon=\varepsilon_{j_i}\searrow0.
\ee
\end{lemma}
\begin{proof}
By Lemma \ref{lemma-4.13}, for  $p>2n$,   there    exist  $t_1=t_1(p)>0$ and $\varepsilon_1(p)\in(0,1)$   such that for $\varepsilon\in(\varepsilon_j)_{j\in\mathbb{N}}$ with $\varepsilon<\varepsilon_1(p)$, the  solution of \eqref{reg-2.11}  satisfies
\be\label{int-4.68}
\|u_{\varepsilon}(\cdot,t)\|_{L^p(\Omega)}+\|v_{\varepsilon}(\cdot,t)\|_{L^p(\Omega)}\leq C_1,  \ \ \ \forall t\geq t_1.
\ee
 We use the variation-of-constants formula for $w_{\varepsilon}$ to write
$$
w_{\varepsilon}(\cdot,t)=e^{t( \Delta-1)}w_{\varepsilon}(\cdot,t_1)-\int^t_{t_1}e^{(t-s)( \Delta-1)}(\alpha F_{\varepsilon}(
u_{\varepsilon})+\beta F_{\varepsilon}(v_{\varepsilon})-1)w_{\varepsilon}(\cdot,s)ds, \  t\geq t_1.
$$
Since $|F_{\varepsilon}(s)|\leq |s|$ and $w_{\varepsilon}\leq\|w_0\|_{L^{\infty}(\Omega)}$, we employ the well-known smoothing $L^p$-$L^q$-estimates for the Neumann heat  semigroup  $\{e^{t\Delta}\}_{t\geq0}$  (c.f. \cite{Hor-Win, Win10-JDE}) to infer
\be\label{int-4.69}\begin{split}
\|\nabla w_{\varepsilon}(\cdot,t)\|_{L^p(\Omega)}&\leq C_2\| w_{\varepsilon}(\cdot,t_1)\|_{L^{\infty}(\Omega)} +C_2\int^t_{t_1}\left[1+(t-s)^{-\frac{1}{2}}\right]e^{-(t-s)}\\
& \ \ \ \times\left(\alpha\|u_{\varepsilon}(\cdot,s)\|_{L^p}+\beta\|v_{\varepsilon}(\cdot,s)\|_{L^p}+1\right)ds\\
&\leq C_3, \ \ \  t\geq t_1+1.
\end{split}
\ee
Given \eqref{int-4.68} and \eqref{int-4.69}, H\"{o}lder's inequality enables us to infer that
\be\label{uv-gradw-est}
\begin{split}
&\|u_{\varepsilon}(\cdot,t)\nabla w_{\varepsilon}(\cdot,t)\|_{L^{\frac{p}{2}}}+\|v_{\varepsilon}(\cdot,t)\nabla w_{\varepsilon}(\cdot,t)\|_{L^{\frac{p}{2}}}\\[0.2cm]
&\leq \left(\|u_{\varepsilon}(\cdot,t)\|_{L^p}+\|v_{\varepsilon}(\cdot,t)\|_{L^p}\right)\|\nabla w_{\varepsilon}(\cdot,t)\|_{L^p}\leq C_4, \ \ \  t\geq t_2:=t_1+1.
\end{split}
\ee
Applying the variation-of-constants formula for $u_{\varepsilon}$, we have
$$
u_{\varepsilon}(\cdot,t)=e^{t\Delta}u_{\varepsilon}(\cdot,t_2)-\chi_1\int^t_{t_2}e^{(t-\tau)\Delta}\nabla\cdot(u_{\varepsilon}
F'_{\varepsilon}(u_{\varepsilon})\nabla w_{\varepsilon})(\cdot,\tau)d\tau, \  t\geq t_2.
$$
This together with \eqref{int-4.68} and \eqref{uv-gradw-est} allows us to  find  $\theta\in(0,1)$, $\gamma\in(0,1)$ and $q>1$ suitably  large such that $2\theta-\frac{n}{q}>0$,
\be\label{int-4.70}
\|A^\theta u_{\varepsilon}(\cdot,t)\|_{L^q}\leq C_5,  \ \ \ \forall t>t_3:=t_2+1
\ee
and, for any   $t,s\geq t_3$ such that $|t-s|\leq1$,
\be\label{int-4.71}
\|A^\theta u_{\varepsilon}(\cdot,t)-A^\theta u_{\varepsilon}(\cdot,s)\|_{L^q}\leq C_6|t-s|^\gamma,
\ee
where $A^\theta$ denotes the fractional power of the realization of $-\Delta+1$ in $L^q(\Omega)$ under homogeneous Neumann boundary conditions. By the continuous embedding  $D(A^\theta)\hookrightarrow C^\sigma$ for all $\sigma\in(0,2 \theta-\frac{n}{q})$ (cf.  \cite{Henry, Hor-Win, Xiangjde}),  \eqref{int-4.70} and \eqref{int-4.71}, we know that $(u_{\varepsilon})_{\varepsilon\in (\varepsilon_j)_{j\in\mathbb{N}}}$ is bounded in both $L^{\infty}(\Omega\times(t_3,\infty))$ and in $C^{\sigma,\frac{\sigma}{2} }_{loc}(\overline{\Omega}\times[t_3,\infty))$ for some $\sigma\in(0,1)$. Analogously, we also have $(v_{\varepsilon})_{\varepsilon\in (\varepsilon_j)_{j\in\mathbb{N}}}$ is bounded in both $L^{\infty}(\Omega\times(t_3,\infty))$ and in $C^{\sigma,\frac{\sigma}{2} }_{loc}(\overline{\Omega}\times[t_3,\infty))$ for some $\sigma\in(0,1)$. Thus,   the standard parabolic Schauder estimates
\cite{LSU-bk} applied to the third equation in \eqref{reg-2.11} yield boundedness of $(w_{\varepsilon})_{\varepsilon\in(\varepsilon_j)_{j \in\mathbb{N}}}$ in both $L^{\infty}((t_4,\infty);C^{2+\sigma}(\overline{\Omega}))$ and in $C^{2+\sigma,1+\frac{\sigma}{2} }_{loc}(\overline{\Omega}\times[t_4,\infty))$ with $t_4=t_3+1$. This, in turn, by a similar argument, we also obtain the boundedness of $(u_{\varepsilon})_{\varepsilon \in (\varepsilon_j)_{j\in\mathbb{N}}}$ and $(v_{\varepsilon})_{\varepsilon\in (\varepsilon_j)_{j\in\mathbb{N}}}$ in both $L^{\infty}((t_5,\infty);C^{2+\sigma'}(\overline{\Omega}))$ and in $C^{2+\sigma',1+\frac{\sigma'}{2} }_{loc}(\overline{\Omega}\times[t_5,\infty))$ for some $\sigma'\in(0,1)$ and $t_5=t_4+1$, and thus \eqref{int-4.66} follows. Finally, an application of the Arzel\`{a}-Ascoli theorem implies \eqref{int-4.67}.
\end{proof}

\begin{lemma}\label{lemma-4.15} For $n\in\{3, 4, 5\}$, the global weak solution of \eqref{PPT} constructed from Lemma \ref{lemma-4.8} satisfies
\be\label{int-4.75}
u(\cdot,t)\rightarrow\bar{u}_0 \ \ and \ \ v(\cdot,t)\rightarrow\bar{v}_0 \ \ in \ L^{\infty}(\Omega) \ as \ t\rightarrow\infty,
\ee
where $\bar{u}_0$ and $\bar{v}_0$ are the average of $u_0$ and $v_0$ over $\Omega$, respectively.
\end{lemma}
\begin{proof}
With the information provided by Lemmas \ref{lemma-4.4},  \ref{lemma-4.6}  and \ref{lemma-4.14}, we can readily extend and adapt the arguments in  \cite[Lemma 7.2]{TW12-JDE} to derive \eqref{int-4.75}. Indeed, let us assume to the contrary there exists a sequence of $t_k\rightarrow \infty$ such that
\be\label{ulimt-1}
d:=\inf_{k\in\mathbb{N}}\left\|u(\cdot,t_k)-\bar{u}_0\right\|_{L^\infty}>0,
\ee
where we with no loss of generality can assume that all $t_k>T$ and $1+t_k\leq t_{k+1}$ as $T$ provided by Lemma \ref{lemma-4.14}.  In light of \eqref{int-4.66} and \eqref{int-4.67}, $(u(\cdot, t_k))_{k\in\mathbb{N}}$ is relatively compact, and then by the Arzel\`{a}-Ascoli theorem, we may assume for convenience, for some nonnegative function $u_\infty$, that
\be\label{ulimt-2}
u(\cdot, t_k)\rightarrow u_\infty \ \ in \ L^\infty(\Omega) \ as \ k\rightarrow\infty.
\ee
In the sprit of  Lemma \ref{lemma-4.11}, we use the Cauchy-Schwarz inequality  to estimate that
\begin{align*}
&\int_{t_k}^{1+t_k} \left\|u_\varepsilon(\cdot, t)-u_\varepsilon(\cdot, t_k)\right\|_{(W^{m,2})^*}^2\\
&=\int_{t_k}^{1+t_k}\left\|\int_{t_k}^tu_{\varepsilon t}(\cdot, s)\right\|_{(W^{m,2})^*}^2\leq  \int_{t_k}^{1+t_k}\left\|u_{\varepsilon t}(\cdot, s)\right\|_{(W^{m,2})^*}^2, \quad  \forall \varepsilon\in (0, 1).
\end{align*}
In view of  \eqref{int-4.72} in Lemma \ref{lemma-4.6}, upon passing to the limit $\varepsilon=\varepsilon_j\searrow0$, this yields
$$
\int_{t_k}^{1+t_k} \left\|u(\cdot, t)-u(\cdot, t_k)\right\|_{(W^{m,2})^*}^2\leq  \int_{t_k}^{1+t_k}\left\|u_t(\cdot, s)\right\|_{}^2\rightarrow 0 \text{  as }  k\rightarrow \infty.
$$
Since $L^\infty(\Omega)\hookrightarrow \left(W^{m,2}(\Omega)\right)^*$ due to $m>\frac{n}{2}{+1}$,  it then follows from \eqref{ulimt-2} that
$$
\int_{t_k}^{1+t_k} \left\| u(\cdot, t_k)-u_\infty\right\|_{(W^{m,2})^*}^2\rightarrow 0 \text{  as }  k\rightarrow \infty.
$$
Hence, we use  triangle inequality to deduce from the above two estimates that
\be\label{ulimt-3}
\int_{t_k}^{1+t_k} \left\| u(\cdot, t)-u_\infty\right\|_{(W^{m,2})^*}^2\rightarrow 0 \text{  as }  k\rightarrow \infty.
\ee
On the other hand, notice also that $L^\frac{n}{n-1}(\Omega)\hookrightarrow \left(W^{m,2}(\Omega)\right)^*$ due to $m>\frac{n}{2}{+1}$;     The content of Lemma \ref{lemma-4.4} from Lemma \ref{lemma-4.10} ensures that
\be\label{ulimt-4}
\int_{t_k}^{1+t_k}\|u(\cdot, t)-\bar{u}_0\|^2_{(W^{m,2})^*}
\rightarrow 0 \text{  as }  k\rightarrow \infty.
\ee
By uniqueness, it follows from \eqref{ulimt-3} and \eqref{ulimt-4} that $u_\infty\equiv \bar{u}_0$, which  is impossible by  \eqref{ulimt-1} and \eqref{ulimt-2}. This contradiction says that  the $u$-limit in \eqref{int-4.75} is true; similarly, the $v$-limits follows in the same way.
\end{proof}
\begin{proof}[\bf{Proof of eventual smoothness and convergence in convex domains}]\ \
The eventual smoothness and boundedness of weak solutions result immediately
from \eqref{int-4.66}  and  \eqref{int-4.67} in Lemma \ref{lemma-4.14}. The convergence of weak solutions as in \eqref{uvw-cov-zt} follows from Lemmas \ref{lemma-4.14} and \ref{lemma-4.15}.
\end{proof}

\textbf{Acknowledgments}   G. Ren was supported by the National Natural Science Foundation of China (No.12001214) and the Postdoctoral Science Foundation (Nos. 2020M672319, 2020TQ0111). T. Xiang was funded by the National Natural Science Foundation of China (Nos. 12071476  and 11871226).

\end{document}